\documentclass{amsart}

\usepackage{amssymb}
\usepackage[hidelinks]{hyperref}
\usepackage[textsize=small]{todonotes}
\usepackage{tikz}
\usetikzlibrary{decorations.text}
\usepackage{xcolor}
\usepackage[normalem]{ulem}

\theoremstyle{plain} 
\newtheorem{theorem}{Theorem}[section] 
\newtheorem{lemma}[theorem]{Lemma} 
\newtheorem{corollary}[theorem]{Corollary} 
\newtheorem{proposition}[theorem]{Proposition}

\theoremstyle{definition} 
\newtheorem{definition}[theorem]{Definition}

\newtheorem{question}[theorem]{Question}  
 
\theoremstyle{remark} 
\newtheorem*{remark}{Remark}

\newcommand{\bR}{\mathbb{R}} 
\newcommand{\R}{\bR}

\newcommand{\bN}{\mathbb{N}} 
\newcommand{\N}{\bN}
\newcommand{\cA}{\mathcal{A}} 
\newcommand{\cB}{\mathcal{B}}

\newcommand{\cF}{\mathcal{F}} 
\newcommand{\cI}{\mathcal{I}} 
\newcommand{\I}{\cI} 
\newcommand{\cJ}{\mathcal{J}} 
\newcommand{\J}{\cJ} 
 
\newcommand{\cM}{\mathcal{M}} 
\newcommand{\cP}{\mathcal{P}}

\newcommand{\cS}{\mathcal{S}}  
\newcommand{\cU}{\mathcal{U}}  
\newcommand{\cW}{\mathcal{W}} 
\newcommand{\cZ}{\mathcal{Z}}  
\newcommand{\continuum}{\mathfrak{c}}
\newcommand{\eu}{\mathcal{EU}} 
\renewcommand{\subset}{\subseteq}

\DeclareMathOperator{\dom}{dom}
\DeclareMathOperator{\ran}{ran}
\DeclareMathOperator{\Exh}{Exh}
\DeclareMathOperator{\fin}{Fin} 
\DeclareMathOperator{\Fin}{Fin}

\DeclareMathOperator{\exh}{Exh}

\begin{document}

\title{Path of pathology}

\author{Rafa\l{} Filip\'{o}w}
\address[R.~Filip\'{o}w]{Institute of Mathematics\\ Faculty of Mathematics, Physics and Informatics\\ University of Gda\'{n}sk\\ ul. Wita Stwosza 57\\ 80-308 Gda\'{n}sk\\ Poland}
\email{Rafal.Filipow@ug.edu.pl}
\urladdr{http://mat.ug.edu.pl/~rfilipow}

\author{Jacek Tryba}
\address[J.~Tryba]{Institute of Mathematics,
Faculty of Mathematics, Physics and Informatics,
University of Gda\'{n}sk,
ul. Wita Stwosza 57,
80-308 Gda\'{n}sk,
Poland}
\email{Jacek.Tryba@ug.edu.pl}

\date{\today}



%-------------------------------------------------------------------------------
% Abstract 
%-------------------------------------------------------------------------------

\begin{abstract}
We present a few results about (non)pathology of submeasures and ideals.
\end{abstract}

%%%%%%%%%%%%%%%%%%%%%%%%%%%%%%%%%%%%%%%%
%%%
%%%%%%%%%%%%%%%%%%%%%%%%%%%%%%%%%%%%%%%%

\maketitle

% \setcounter{tocdepth}{1}
% \tableofcontents

%%%%%%%%%%%%%%%%%%%%%%%%%%%%%%%%%%%%%%%%
%%%
%%%%%%%%%%%%%%%%%%%%%%%%%%%%%%%%%%%%%%%%

\section{Introduction}

The paper consists of  two parts. In the first part (Sections~\ref{sec:def-of-pathol}--\ref{sec:degree-of-pathol}) 
we examine  various kinds of (non)pathological submeasures that are considered in the literature, whereas the  second part 
(Sections~\ref{sec:def-of-path-ideal}--\ref{sec:non-borel-intersection})
is devoted to (non)pathological ideals.
Below, we briefly describe the content of each section. 

In Section~\ref{sec:def-of-pathol}, we provide various definitions of (non)pathological submeasures that are considered in the literature and in Section~\ref{sec:properties-of-pathol}  we summarize basic properties of these  kinds of submeasures.
In Sections~\ref{sec:examples-of-nonpathol} and \ref{sec:examples-of-pathol} we present some examples of (non)pathological submeasures. 
The relationships and examples  from these  sections are shown in  concise graphical form
in  Figure~\ref{fig:pathology-vs-lsc-for-submeasure}.

In Section~\ref{sec:degree-of-pathol}, we examine various kinds of  degrees of pathology
of submeasures that are considered in the literature. The results from this section  are summarized in Table~\ref{tab:P-like-degrees}.

The second part of the paper starts with  Section~\ref{sec:def-of-path-ideal} where we present definitions, examples and characterizations of nonpathological ideals. 

In Section~\ref{sec:matrix}, we prove a characterization of ideals that can be represented as the intersections of matrix summability ideals. The characterization is expressed in terms of the Kat\v{e}tov order and restrictions of the ideal of asymptotic density zero sets. This characterization resembles a characterization of pathological analytic P-ideals given by Hru\v{s}\'{a}k in \cite[Corollary~5.6]{hrusak-katetov}. 
As a by-product, we obtain that the Solecki ideal is pathological which answers a question of Mart\'{i}nez, Meza-Alc\'{a}ntara and  Uzc\'{a}tegui \cite[Question~3.13]{martinez2022pathology}. 

In Section~\ref{sec:GDI}, 
we  solve a problem of Borodulin-Nadzieja and Farkas \cite[Problem~4.3]{MR4124855}. Namely, we  show  that the ideal 
of  exponential density  zero sets is 
 a nonpathological analytic P-ideal which is not a special variant of a density ideal generated by a family of measures. 

In Section~\ref{sec:Fsigma-and-degree-of-pathol}, 
we answer two more   questions of Mart\'{i}nez, Meza-Alc\'{a}ntara and  Uzc\'{a}tegui \cite[Questions~3.6 and 3.10]{MR4797308}. Since the second question is quite technical, we describe here only the first one.
We 
construct  a   submeasure $\phi$ which has infinite degree of pathology but the ideal $\Fin(\phi)$ is nonpathological and the degrees of pathology of all restrictions of $\phi$ to sets from the ideal $\Fin(\phi)$ are finite.
In the same section we also characterize nonpathological $F_\sigma$ ideals as those which can be represented as the intersection of summable ideals. This result strengthen a known theorem in which summable ideals are replaced by matrix ideals. 

In Section~\ref{sec:vdW}, we prove that the van der Waerden ideal (which consists of those subsets of integers which do not contain arithmetic progressions of  arbitrary finite length) is  nonpathological.

In Section~\ref{sec:JNP}, we consider the  Josefson-Nissenzweig property which has been recently introduced and examined by Marciszewski and Sobota \cite{marciszewski-sobota}.
We prove that the disjoint union of topological spaces has the Josefson-Nissenzweig property if and only if at least one of these spaces has this property. As a corollary, we obtain  that a variant of Mr\'{o}wka-Isbell space defined with the aid of almost disjoint family $\cA$ of sets which are not in  a given ideal $\I$ has the Josefson-Nissenzweig property if and only if the restriction of $\I$ to some set from the family $\cA$ extends to a matrix ideal.

In Section~\ref{sec:non-borel-intersection}, we  alter a result of  
Hru\v{s}\'{a}k and Garc\'{i}a Ferreira \cite[Theorem~3.9]{MR2017358} to show that consistently for every tall ideal $\I$ there exists a maximal almost disjoint family $\cA\subseteq\I$ such that the ideal $\I(\cA)$ generated by $\cA$ is homogeneous in the sense of Kat\v{e}tov order. Next, we show that this ideal $\I(\cA)$ is a non-Borel ideal which is the intersection of matrix ideals. This answer the question posed by the authors in \cite[Question~3]{ft-mazur}.

%%%%%%%%%%%%%%%%%%%%%%%%%%%%%%%%%%%%%%%%
%%%
%%%%%%%%%%%%%%%%%%%%%%%%%%%%%%%%%%%%%%%%

\section{Preliminaries}

We identify an ordinal number $\alpha$  with the set of all ordinal numbers less than $\alpha$. 
In particular, the  smallest infinite ordinal number $\omega=\{0,1,\dots\}$ is equal to the set of all natural numbers $\N$, and each natural number $n = \{0,\dots,n-1\}$ is equal  to the set of all natural numbers less than $n$.
Using this identification, we can for instance write $n\in k$ instead of $n<k$ and $n<\omega$ instead of $n\in \omega$ or $A\cap n$ instead of $A\cap \{0,1,\dots,n-1\}$. 

%%%%%%%%%%%%%%%%%%%%%%%%%%%%%%%%%%%%%%%%
%%%
%%%%%%%%%%%%%%%%%%%%%%%%%%%%%%%%%%%%%%%%

\subsection{Ideals}

\begin{definition}
\label{def:ideal}
An \emph{ideal} on a nonempty set $X$ is a  family $\I\subseteq\cP(X)$ that satisfies the following properties:
\begin{enumerate}
\item $\emptyset\in \I$ and $X\not\in\I$,
\item if $A,B\in \I$ then $A\cup B\in\I$,
\item if $A\subseteq B$ and $B\in\I$ then $A\in\I$.
\end{enumerate}
In the second part of the paper (starting at page \pageref{part:ideals}), we alter the definition of an ideal. Namely, we will additionally assume there that  
an ideal $\I$ on $X$ has to contain all \emph{finite subsets of $X$}.
\end{definition}

For an ideal $\I$,  we write $\I^+=\{A\subseteq \N: A\notin\I\}$ and call it the \emph{coideal of $\I$}, 
 we also write $\I^*=\{A\subseteq \N: \N\setminus A\in\I\}$ and call it the \emph{dual filter of $\I$}.
By ${\fin}$ we mean the ideal of all finite subsets of $\N$.
An ideal $\I$ is
a \emph{P-ideal} if for every countable family $\cA\subset\I$ there is $B\in\I$ such that $A\setminus B$ is finite for every $A\in\cA$.
An ideal $\I$ is \emph{tall} (a.k.a. \emph{dense}) if for every infinite $A\subseteq X$ there is an infinite $B\in\I$ such that $B\subseteq A$.
For an ideal $\I$ and a set $A\not\in\I$ we define the \emph{restriction of the ideal}
$\I$ to $A$ by 
$\I\upharpoonright A= \{B\cap A:  B\in\I\}$.
An ideal $\I$ on $X$ is called \emph{maximal} if there is no ideal $\J$ on $X$ with $\I\subseteq \J$ and $\I\neq\J$. It is known that an ideal $\I$ on $X$ is maximal if and only if for each $A\subseteq X$ either $A\in \I$ or $X\setminus A\in \I$. 
Dual filters of maximal ideals are called \emph{ultrafilters}.

Let $\I$ and $\J$ be ideals on $X$ and $Y$ respectively.
We say that $\I$ and $\J$ are \emph{isomorphic} (in short $\I\approx\J$) 
if there exists a bijection $f:X\rightarrow Y$ such that $A\in\I\iff f[A]\in\J$ for every $A\subset X$.
% We say that \emph{$\I$ is below $\J$ in Rudin-Blass order} (in short $\I\leq_{RB}\J$) if there is finite-to-one function $f:Y\rightarrow X$ such that for every $A\subseteq X$ we have $A\in\I$ if and only if $f^{-1}[A]\in \J$.
% We say that $\I$ and $\J$ are \emph{Rudin-Blass equivalent} (in short $\leq_{RB}$-equivalent) if $\I\leq_{RB}\J$ and $\J\leq_{RB}\I$.
We say that \emph{$\J$ is below $\I$ in  Kat\v{e}tov order} (in short $\J\leq_{K}\I$) if there is a function $f:X\rightarrow Y$ such that $A\in\J \implies f^{-1}[A]\in\I$
for every $A\subseteq X$. 
We say that $\I$ and $\J$ are \emph{Kat\v{e}tov equivalent} (in short $\I\approx_K\J$
or $\leq_{K}$-equivalent) if $\I\leq_{K}\J$ and $\J\leq_{K}\I$.
We say that $\I$ is \emph{K-homogeneous} if $\I$ and $\I\restriction A$ are Kat\v{e}tov equivalent for every $A\notin \I$.

By identifying sets of natural numbers with their characteristic functions,
we equip $\cP(\omega)$ with the topology of the Cantor space $\{0,1\}^\omega$ and therefore
we can assign topological complexity to ideals on $\omega$.
In particular, an ideal $\I$ is Borel ($F_\sigma$, analytic, resp.) if $\I$ is Borel ($F_\sigma$, analytic, resp.) as a subset of the Cantor space.

%%%%%%%%%%%%%%%%%%%%%%%%%%%%%%%%%%%%%%%%
%%%
%%%%%%%%%%%%%%%%%%%%%%%%%%%%%%%%%%%%%%%%

\subsection{Submeasures}
\label{sec:Measures}

A map $\phi:\cP(X)\to[0,\infty]$ is a \emph{measure} (\emph{$\sigma$-measure}, resp.) on a nonempty set $X$  if $\phi(\emptyset)=0$ and $\phi$ is finitely (countably, resp.) additive. 

For a function $f:\omega\to[0,\infty)$,  we define \label{def:measure-for-summable-ideal}
a   $\sigma$-measure $\mu_f$ on $\omega$  by 
$$\mu_f(A)= \sum_{n\in A} f(n) .$$
It is not difficult to see that 
$\phi$ is a $\sigma$-measure on $\omega$ $\iff$
$\phi =\mu_f$ for some function $f:\omega\to[0,\infty]$.

For an ideal $\I$ on $\omega$, we define a measure $\delta_\I^\infty$ on $\omega$
by 
$$
\delta_\I^\infty(A)=
\begin{cases}
    0&\text{if $A\in \I$,}\\
    \infty&\text{otherwise,}
\end{cases}
$$
and for a  maximal ideal $\I$, we define a measure $\delta_\I$ by  
$$\delta_\I(A)=
\begin{cases}
    0&\text{if $A\in \I$,}\\
    1&\text{otherwise.}
\end{cases}
$$

To obtain more examples of measures, one can use the so called \emph{ultrafilter limits} (see e.g.~\cite{MR1845008}). Namely, if 
$\cU$ is an ultrafilter  on $\omega$
and 
$\mu_n$ is a  measure on $X$ for each $n\in\omega$, then 
$$\nu(A) = \lim_{n\in \cU}\mu_n(A)$$
is a measure on $X$.

\begin{definition}
A map $\phi:\cP(X)\to[0,\infty]$ is a \emph{submeasure} on a set $X$ if 
\begin{enumerate}
\item $\phi(\emptyset)=0$,
\item if $A\subset B$ then $\phi(A)\leq \phi(B)$,
\item $\phi(A\cup B) \leq \phi(A)+\phi(B)$.
\end{enumerate}
\end{definition}

We say that a submeasure $\phi$ is \emph{lower semicontinuous} (in short: \emph{lsc}) if 
$$\phi(A)=\sup\{\phi(F):\text{$F$ is a finite subset of $A$}\}$$ 
for each $A\subseteq X$. If $X=\omega$, then $\phi$ is lsc $\iff$  
$\phi(A)=\lim_{n\to\infty} \phi(A\cap n).$

Every measure is a submeasure, and  a measure 
$\phi$ is a $\sigma$-measure $\iff$ $\phi$ is lower semicontinuous.
Moreover,  if $\phi$ is a nonzero submeasure on $X$, then 
    $$\cZ_\phi=\{A\subseteq X: \phi(A)=0\}$$ is an ideal on $X$,
    whereas if $\I$ is an ideal on $X$, then 
$$\delta_\I(A)=
\begin{cases}
    0&\text{if $A\in \I$,}\\
    1&\text{otherwise}
\end{cases}
$$
is a submeasure on $X$ with $\cZ_{\delta_\I}=\I$.

To obtain more examples of submeasures on $\omega$, one can observe that  if $\phi_n$ is a  submeasure on $X$ for every $n\in \omega$, then 
$$\phi(A) = \limsup_{n\to\infty}\phi_n(A)$$
is a submeasure on $X$, 
and if  $\cS$ is a family of submeasures (lsc submeasures, resp.) on $X$, then
$$\psi(A)=\sup\{\phi(A):\phi\in \cS \}$$ 
is a  submeasure (lsc submeasures, resp.) on $X$.

%%%%%%%%%%%%%%%%%%%%%%%%%%%%%%%%%%%%%%%%
%%%
%%%%%%%%%%%%%%%%%%%%%%%%%%%%%%%%%%%%%%%%

\part{(Non)pathological submeasures}
\label{part:submeasures}

In this part we present thorough analysis of various  kinds   of (non)pathological submeasures that are considered in the literature. 

%%%%%%%%%%%%%%%%%%%%%%%%%%%%%%%%%%%%%%%%
%%%
%%%%%%%%%%%%%%%%%%%%%%%%%%%%%%%%%%%%%%%%

\section{Definition of (non)pathological submeasures}
\label{sec:def-of-pathol}

For two submeasures $\phi$ and $\psi$ on a set $X$, we say that $\psi$ \emph{dominates} $\phi$ and write
$\phi\leq \psi$ if $\phi(A)\leq \psi(A)$ for every $A\subseteq X$.

For a submeasure $\phi$ on $X$, we define 
two submeasures   $\widehat{\phi}$ and $\widehat{\phi}_\sigma$ on $X$ 
by
\begin{equation*}
    \begin{split}
\widehat{\phi}(A) 
& = 
\sup\{\mu(A) : \text{$\mu$ is a \ \ measure on $X$  such that $\mu \leq \phi$}\},
\\
\widehat{\phi}_\sigma(A) 
&= 
\sup\{\mu(A) : \text{$\mu$ is a $\sigma$-measure on $X$ such that $\mu \leq \phi$}\}.
    \end{split}
\end{equation*}
% It is not difficult  to see that 
% both $\widehat{\phi}$ and $\widehat{\phi}_\sigma$
% are submeasures and 
% $\widehat{\phi}_\sigma\leq \widehat{\phi}\leq \phi$.

\begin{definition}
\label{def:nonpathological-submeasure}
Let $\phi$ be a submeasure on $X$.
\begin{enumerate}
    \item 
A submeasure $\phi$ is called  \emph{nonpathological}  
if 
$\phi=\widehat{\phi}$,  
% i.e.
% $$\phi(A) = \sup\{\mu(A) : \text{$\mu$ is a measure on $X$  such that $\mu \leq \phi$}\}$$ 
% for each $A\subseteq X$, 
% and 
otherwise $\phi$ is called \emph{pathological} (Farah, \cite[p.~21]{MR1711328}).
\item A submeasure $\phi$ is called  \emph{$\sigma$-nonpathological} 
if 
$\phi=\widehat{\phi}_\sigma$,  
% i.e.
% $$\phi(A)= \sup\{\mu(A) : \text{$\mu$ is a $\sigma$-measure on $X$ such that $\mu \leq \phi$}\}$$ for each $A\subseteq X$, 
% and 
otherwise $\phi$ is  called   \emph{$\sigma$-pathological}. 
\end{enumerate}
\end{definition} 

The following remark shows that there is a mess with definitions of (non)pathological   submeasures.  

\begin{remark}\ 
\begin{enumerate}
    \item 
In \cite{MR0419712}, Tops\o{}e introduced the notion of $\varepsilon$-pathological submeasures
and one can show  that a submeasure $\phi$ is pathological (in the sense of Farah) $\iff$ 
there exists  $A\subseteq X$ and $\varepsilon<1$
such that the restriction $\phi\restriction\cP(A)$ is  $\varepsilon$-pathological (in the sense of Tops\o{}e).

\item 
To make life a bit confusing, we can also find out that in \cite{MR0419712} Tops\o{}e  has a  definition  of a nonpathological submeasure which is not coherent with the notion introduced by Farah, namely a submeasure $\phi$ is nonpathological in the sense of Tops\o{}e if $\widehat{\phi}=0$ (i.e.~there is no nonzero measure dominated by $\phi$). 
It is easy to see that pathology in the sense of Tops\o{}e is stronger than pathology in the sense of Farah (Proposition~\ref{prop:pathology-Topsoe}(\ref{prop:pathology-Topsoe:stronger-then-Farah})).
On the other hand, the pathology in the sense of Tops\o{}e  is useless in the realm of \emph{lower semicontinuous} submeasure  as is shown in Proposition~\ref{prop:pathology-Topsoe}(\ref{prop:pathology-Topsoe:lsc}).

\item To make life even more complicated, in \cite{MR4797308} (see also  \cite[p.~3]{martinez2022fsigma} and \cite[p.~4]{martinez2022pathology}), the authors define a submeasure $\phi$ to be nonpathological if $\phi=\widehat{\phi}_\sigma$. 
Obviously, $\phi=\widehat{\phi}_\sigma$ implies $\phi=\widehat{\phi}$. 
The converse implication does not hold in general (see Proposition~\ref{prop:properties-of-delta-submeasures}(\ref{prop:properties-of-delta-submeasures:delta-ideal})), but it holds  for  \emph{lower semicontinuous} submeasures (see Proposition~\ref{prop:properties-of-phi-hat}(\ref{prop:properties-of-phi-hat:all-sets-for-lsc})).
Moreover, there exists a pathological lower semicontinuous submeasure $\phi$ such that $\widehat{\phi}_\sigma\neq \widehat{\phi}$ (see Proposition~\ref{prop:pathological-lsc-submeasures}).
\end{enumerate}
\end{remark}

The following proposition (a sort of the extreme value theorem) seems to be  a folklore, but we decided to include the proof for the sake of completeness because we did not find any reference which contains a proof. 

\begin{proposition}
\label{prop:phi-hat-vs-measure}
     For every 
 submeasure  $\phi$ on $X$
 and every 
$A\subseteq X$ there exists a measure $\mu$ on $X$ such that $\mu\leq\phi$ and $\widehat{\phi}(A)=\mu(A)$.\label{prop:phi-hat-vs-measure:measure} 
In particulary, 
if $\phi$ is nonpathological, then for every $A\subseteq X$ there exists a measure $\mu$ on $X$ such that $\mu\leq\phi$ and $\phi(A)=\mu(A)$.
\end{proposition}

\begin{proof}
Take any $A\subseteq X$.
We have two cases: $\widehat{\phi}(A)=\infty$ and $\widehat{\phi}(A)<\infty$.

\emph{Case: $\widehat{\phi}(A)=\infty$.}
We define a function $\mu:\cP(X)\to [0,\infty]$  by 
$$\mu(B) = \begin{cases}
    \infty& \text{if $\widehat{\phi}(B)=\infty$,}\\
    0&\text{otherwise.}
\end{cases}$$
Then $\mu$ is a measure on $X$,  $\mu\leq \phi$ and $\mu(A)=\widehat{\phi}(A)$, so we are done.

\emph{Case: $\widehat{\phi}(A)<\infty$.}
Let $\cM$ be a family of all measures $\mu:\cP(A)\to [0,\widehat{\phi}(A)]$ such that $\mu(B)\leq \phi(B)$ for every $B\subseteq A$. We claim that $\cM$ is a closed set in the space $[0,\widehat{\phi}(A)]^{\cP(A)}$ with the product topology.
Indeed, take any  $f\in [0,\widehat{\phi}(A)]^{\cP(A)} \setminus \cM$.
Then we have two cases: 
(a) $f$ is not a measure 
or 
(b) $f$ is a measure. 

Case (a). 
Then we can find  disjoint sets $A_0,A_1\subseteq A$ with $f(A_0\cup A_1)\neq f(A_0)+f(A_1)$. 
For  $\varepsilon = |f(A_0\cup A_1) - (f(A_0)+f(A_1))|/4>0$, we define 
$$U = \{g\in [0,\widehat{\phi}(A)]^{\cP(A)}: |g(A_0\cup A_1) - f(A_0 \cup A_1)|<\varepsilon , |g(A_i)-f(A_i)|<\varepsilon \text{ for $i<2$}\}.$$
Then $U$ is an open neighborhood of $f$ and $U\cap \cM=\emptyset$. 

Case (b). 
Since $f\notin \cM$, we can find  $B\subseteq A$ with $f(B)>\phi(B)$.
For $\varepsilon = (f(B)-\phi(B))/3>0$, we define 
$$U = \{g\in [0,\widehat{\phi}(A)]^{\cP(A)}: |g(B) - f(B)|<\varepsilon\}.$$
Then $U$ is an open neighborhood of $f$ and $U\cap \cM=\emptyset$. 

In both cases, we find an open neighborhood of $f$ which is disjoint from $\cM$, thus $\cM$ is closed.
Since $[0,\widehat{\phi}(A)]^{\cP(A)}$ is a compact space, we obtain that $\cM$ is compact as well.
For each  $n\in \omega$, we define 
$$C_n = \{\mu\in \cM: |\mu(A)-\widehat{\phi}(A)|\leq 1/(n+1)\}.$$
Then $C_n$ is a decreasing sequence of closed sets, so its intersection is nonempty. Let  $\mu\in \bigcap_{n\in \omega}C_n\subseteq\cM$. Then $\mu$ is a measure such that $\mu\leq \phi$ and $\mu(A)=\widehat{\phi}(A)$, so we are done.
\end{proof}

%%%%%%%%%%%%%%%%%%%%%%%%%%%%%%%%%%%%%%%%
%%%
%%%%%%%%%%%%%%%%%%%%%%%%%%%%%%%%%%%%%%%%

\section{Properties of (non)pathological submeasures}
\label{sec:properties-of-pathol}

In Proposition~\ref{prop:properties-of-submeasures}  we summarize basic properties of (non)pathological submeasures and in Sections~\ref{sec:examples-of-nonpathol} and \ref{sec:examples-of-pathol} we provide some examples of (non)pathological submeasures. 
We present these relationships and examples  in concise graphical form
in  Figure~\ref{fig:pathology-vs-lsc-for-submeasure}.

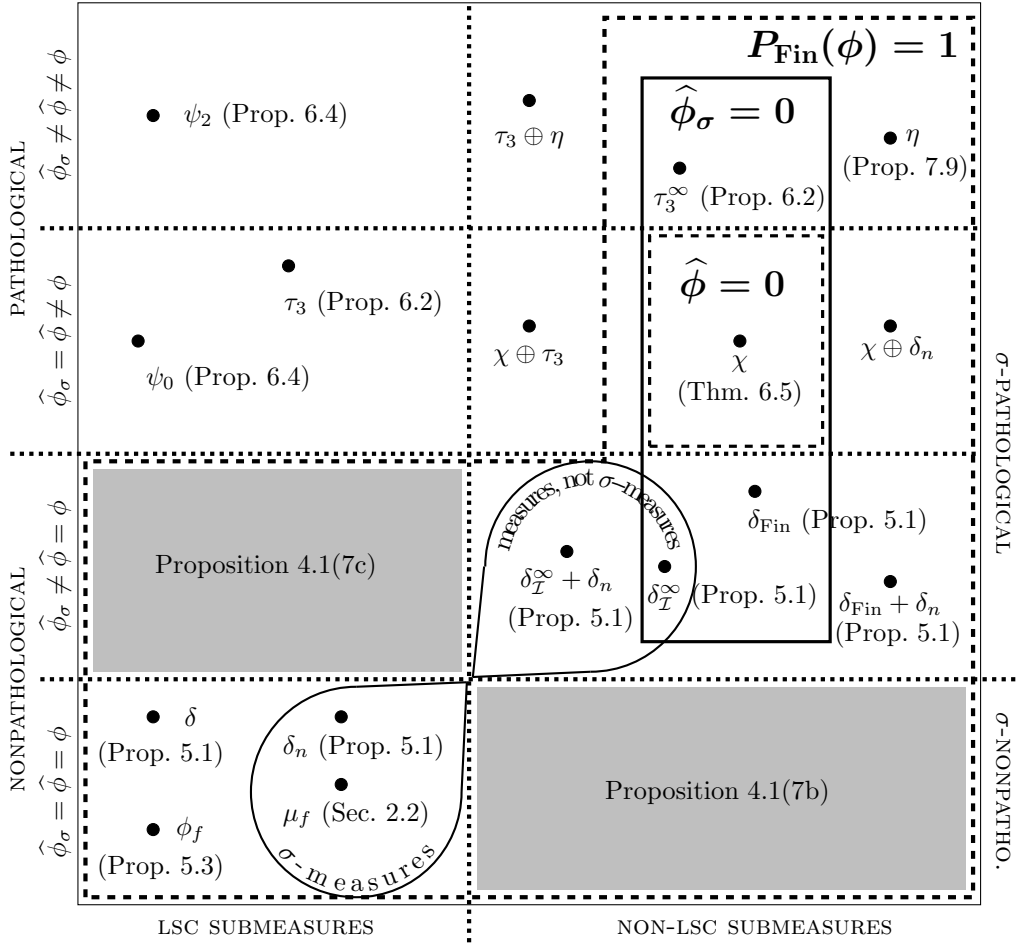
\begin{figure}
\centering
\begin{tikzpicture}
\draw (-6,-6) rectangle (6,6);

\draw[ultra thick,dotted] (-0.8,-6.5) -- (-0.8,6);
\node at (-3.5,-6.3) {\textsc{lsc submeasures}};
\node at (3,-6.3) {\textsc{non-lsc submeasures}};

\draw[ultra thick,dotted] (-6.5,-3) -- (6.5,-3);
\draw[ultra thick,dotted] (-6.9,0) -- (6,0);
\draw[ultra thick,dotted] (-6.5,3) -- (6,3);
\node[rotate=90] at (-6.8,-3) {\textsc{nonpathological}};
% \node[rotate=-90] at (6.5,-4) {\textsc{$\sigma$-nonpathological}};
\node[rotate=-90] at (6.3,-4.5) {\textsc{$\sigma$-nonpatho.}};
\node[rotate=90] at (-6.3,-4.5) {$\widehat{\phi}_\sigma=\widehat{\phi}=\phi$};
\node[rotate=90] at (-6.3,-1.5) {$\widehat{\phi}_\sigma \neq\widehat{\phi}=\phi$};
\node[rotate=90] at (-6.8,3) {\textsc{pathological}};
\node[rotate=-90] at (6.3,0) {\textsc{$\sigma$-pathological}};
\node[rotate=90] at (-6.3,1.5) {$\widehat{\phi}_\sigma=\widehat{\phi}\neq \phi$};
\node[rotate=90] at (-6.3,4.5) {$\widehat{\phi}_\sigma\neq\widehat{\phi}\neq \phi$};

%% \draw[thick] (0.8,-1.5) circle (1.4cm);
\draw[thick] (0.815,-2.9) arc [start angle=-90,end angle=180,radius=1.4]; 
\draw[thick] (-0.75,-2.97) -- (0.815,-2.9);
\draw[thick] (-0.75,-2.97) -- (-0.6,-1.5);
\draw[decorate,decoration={text along path,  reverse path=true, raise=-12pt, text={measures, not {$\sigma$}--measures},text align=center}]
(2.3,-1.5) arc [start angle=0,end angle=180,radius=1.5];

%% \draw[thick] (-2.4,-4.5) circle (1.4);
\draw[thick] (-2.3,-3.1) arc [start angle=90,end angle=360,radius=1.4]; 
\draw[thick] (-0.83,-3.04) -- (-2.3,-3.1);
\draw[thick] (-0.83,-3.04) -- (-0.9,-4.5);
% \draw[decorate,decoration={text along path,  reverse path=true, raise=-12pt, text={{$\sigma$}--measures},text align=center}]
% (-0.8,-4.5) arc [start angle=0,end angle=180,radius=1.9];
\draw[decorate,decoration={text along path, reverse path=false,  raise=-12pt, text={{$\sigma$}-measures},text align=center}]
(-3.2,-4.5) arc [start angle=180,end angle=360,radius=0.9];

\fill [fill=gray!50] (-5.8,-2.9) rectangle (-0.9,-0.2);
\fill [fill=gray!50] (5.8,-5.8) rectangle (-0.7,-3.1);

\node at (-3.5,-1.5) {Proposition~\ref{prop:properties-of-phi-hat}(\ref{prop:properties-of-phi-hat:all-sets-for-lsc:lsc})};

\node at (2.5,-4.5) {Proposition~\ref{prop:properties-of-phi-hat}(\ref{prop:properties-of-phi-hat:all-sets-for-lsc:not-lsc})};

\draw[very thick] (1.5,-2.5) rectangle (4,5);
\node at (2.7,4.6) {\LARGE $\boldsymbol{\widehat{\phi}_\sigma=0}$};

\draw[fill] (2,3.8) circle (0.08cm);
\node[thick] at (2.8,3.4) {$\tau_3^\infty$ (Prop.~\ref{prop:tau-3-infty})};

\draw[fill] (4.8,4.2) circle (0.08cm);
\node[thick] at (5.1,4.2) {$\eta$};
\node[thick] at (5,3.8) {(Prop.~\ref{prop:eta})};

\draw[fill] (0,4.7) circle (0.08cm);
\node[thick] at (0,4.2) {$\tau_3\oplus \eta$};

\draw[very thick, dashed] (1.6,0.1) rectangle (3.9,2.9);
\node at (2.7,2.3) {\LARGE $\boldsymbol{\widehat{\phi}=0}$};

\draw[ultra thick,dashed] (-5.9,-5.9) -- (-5.9,-0.1) -- (1,-0.1) -- (1,5.8) -- (5.9,5.8) -- (5.9,-0.1) -- (5.9,-0.1) -- (5.9,-5.9) -- (-5.9,-5.9);
\node at (4.3,5.4) {\LARGE $\boldsymbol{P_{\fin}(\phi)=1}$};

\draw[fill] (-2.5,-3.5) circle (0.08cm);
\node[thick] at (-2.2,-3.9) {$\delta_n$ (Prop.~\ref{prop:properties-of-delta-submeasures})};

\draw[fill] (-2.5,-4.4) circle (0.08cm);
\node[thick] at (-2.3,-4.8) {$\mu_f$ (Sec.~\ref{sec:Measures})};

\draw[fill] (-5,-3.5) circle (0.08cm);
\node[thick] at (-4.5,-3.5) {$\delta$};
\node[thick] at (-4.9,-4) {(Prop.~\ref{prop:properties-of-delta-submeasures})};

\draw[fill] (-5,-5) circle (0.08cm);
\node[thick] at (-4.5,-5) {$\phi_f$};
\node[thick] at (-4.9,-5.5) {(Prop.~\ref{prop:properties-of-asymptotic-densities})};

\draw[fill] (3,-0.5) circle (0.08cm);
\node[thick] at (4.1,-0.9) {$\delta_{\fin}$ (Prop.~\ref{prop:properties-of-delta-submeasures})};

\draw[fill] (1.8,-1.5) circle (0.08cm);
\node[thick] at (2.7,-1.9) {$\delta_\I^\infty$ (Prop.~\ref{prop:properties-of-delta-submeasures})};

\draw[fill] (4.8,-1.7) circle (0.08cm);
\node[thick] at (4.8,-2) {{$\delta_{\fin}+\delta_n$}};
\node[thick] at (4.9,-2.4) { (Prop.~\ref{prop:properties-of-delta-submeasures})};

\draw[fill] (0.5,-1.3) circle (0.08cm);
\node[thick] at (0.5,-1.7) {{$\delta_\I^\infty+\delta_n$}};
\node[thick] at (0.55,-2.2) { (Prop.~\ref{prop:properties-of-delta-submeasures})};

\draw[fill] (-3.2,2.5) circle (0.08cm);
\node[thick] at (-2.2,2) {$\tau_3$ (Prop.~\ref{prop:tau-3-infty})};

\draw[fill] (-5.2,1.5) circle (0.08cm);
\node[thick] at (-4,1) {$\psi_0$ (Prop.~\ref{prop:pathological-lsc-submeasures})};

\draw[fill] (-5,4.5) circle (0.08cm);
\node[thick] at (-3.5,4.5) {$\psi_2$ (Prop.~\ref{prop:pathological-lsc-submeasures})};

\draw[fill] (2.8,1.5) circle (0.08cm);
\node[thick] at (2.8,1.2) {{$\chi$}};
\node[thick] at (2.8,0.8) { (Thm.~\ref{thm:herer-christensen:phi-hat-equals-zero})};

\draw[fill] (4.8,1.7) circle (0.08cm);
\node[thick] at (4.9,1.4) {$\chi\oplus\delta_n$};

\draw[fill] (0,1.7) circle (0.08cm);
\node[thick] at (0,1.3) {$\chi\oplus\tau_3$};

\end{tikzpicture}

    \caption{Relationships between (non)lsc and (non)pathological nonzero submeasures on $\omega$ (in  gray regions there are no submeasures).}
    \label{fig:pathology-vs-lsc-for-submeasure}
\end{figure}

\begin{proposition}
\label{prop:properties-of-submeasures}
\label{prop:properties-of-phi-hat} 
\label{prop:properties-of-phi-hat-sigma}
\label{prop:pathology-Topsoe}
Let $\phi$ be a \emph{nonzero} submeasure on $\omega$.
\begin{enumerate}
\item $\widehat{\phi}$ and $\widehat{\phi}_\sigma$
are submeasures, and  $\widehat{\phi}_\sigma\leq \widehat{\phi}\leq \phi$.
Consequently, if $\phi$
is $\sigma$-nonpathological, then $\phi$ is nonpathological.\label{prop:properties-of-phi-hat:submeasure} 

\item $\widehat{\phi}_\sigma$ is lower semicontinuous. In general, $\widehat{\phi}$ may be not lower semicontinuous (see Proposition~\ref{prop:properties-of-delta-submeasures}(\ref{prop:properties-of-delta-submeasures:delta-ideal-infinity})), even if  $\phi$ is lsc (see Proposition~\ref{prop:pathological-lsc-submeasures}(\ref{prop:pathological-lsc-submeasures:item-non-lsc-hat})).\label{prop:properties-of-phi-hat:sigma-hat-is-lsc} 
        
        \item 
If $\widehat{\phi}=0$, then $\phi$ is pathological.\label{prop:pathology-Topsoe:stronger-then-Farah}

\item 
\begin{enumerate}
\item If $\phi(\omega)=\infty$, then $\widehat{\phi}\neq 0$.\label{prop:pathology-Topsoe:infinite}

\item If $\phi(F)\neq 0$ for some finite set $F$, then   $\widehat{\phi}_\sigma\neq 0$.
Consequently, if $\phi$ is lower semicontinuous, then   $\widehat{\phi}_\sigma\neq 0$.\label{prop:pathology-Topsoe:lsc}

\item If $\phi(F)=0$  for every finite set $F$, then $\phi$ is not lower semicontinuous and $\widehat{\phi}_\sigma=0\neq \phi$.\label{prop:properties-of-submeasures:vanishes-on-FIN}
\end{enumerate}

\item 
\begin{enumerate}

\item $\widehat{\phi}(F)=\widehat{\phi}_\sigma(F)$ for every finite set $F$.\label{prop:properties-of-phi-hat:finite-sets}

\item If  $\phi(\omega\setminus F)=0$
for some finite set $F$, then $\widehat{\phi}_\sigma = \widehat{\phi}$.\label{prop:properties-of-phi-hat:cofinite-set}
\end{enumerate}

\item \label{prop:properties-of-submeasures:measures}
\begin{enumerate}
    \item 
If $\phi$ is a measure ($\sigma$-measure, resp.), then $\phi$ is nonpathological ($\sigma$-nonpathological, resp.).\label{prop:properties-of-submeasures:measures:item-measure}\label{prop:properties-of-submeasures:measures:item-sigma-measure}

\item If $\phi$ is not lower semicontinuous, then 
$\phi$ is $\sigma$-pathological.\label{prop:properties-of-phi-hat:sigma-nopath-implies-lsc}

    \item 
If $\phi$ is a measure but not a $\sigma$-measure, then 
$\phi$ is nonpathological but $\sigma$-pathological.\label{prop:properties-of-phi-hat:measure-not-sigma-meaasure}  
% $\widehat{\phi}_\sigma\neq \widehat{\phi}=\phi$.

\end{enumerate}

\item 
\begin{enumerate}

\item If $\phi$ is lower semicontinuous, then $\phi$ is  nonpathological $\iff$ $\phi$ is $\sigma$-nonpatho\-logical.\label{prop:properties-of-phi-hat:all-sets-for-lsc}
The assumption that $\phi$ is lower semicontinuous cannot be dropped (see Proposition~\ref{prop:properties-of-delta-submeasures}(\ref{prop:properties-of-delta-submeasures:delta-ideal-infinity})).

\item If 
% $\widehat{\phi}_\sigma \neq \widehat{\phi}=\phi$, 
$\phi$ is nonpathological but $\sigma$-pathological, 
then $\phi$ is not lower semicontinuous.\label{prop:properties-of-phi-hat:all-sets-for-lsc:not-lsc}

\item If $\phi$ is $\sigma$-nonpathological, then $\phi$ is lower semicontinuous.\label{prop:properties-of-phi-hat:all-sets-for-lsc:lsc}

\end{enumerate}

\item 
$\widehat{\phi}$ ($\widehat{\phi}_\sigma$, resp.)  is  the largest nonpathological ($\sigma$-nonpathological, resp.) submeasure dominated by $\phi$.\label{prop:properties-of-phi-hat:nonpathological}\label{prop:properties-of-phi-hat:nonpathological-for-sigma}

\end{enumerate}
\end{proposition}

\begin{proof}
(\ref{prop:properties-of-phi-hat:submeasure})
Straightforward.

(\ref{prop:properties-of-phi-hat:sigma-hat-is-lsc})
Because $\widehat{\phi}_\sigma$ is the supremum of a family of lsc (sub)measures.
% Take any $A\subseteq \omega$ and $\varepsilon>0$.
% Then  there is a $\sigma$-measure $\mu\leq \phi$ such that
% $\mu(A)>\widehat{\phi}_\sigma(A)-\varepsilon$.
% Using lsc of $\mu$, we obtain 
% $$\mu(A) = \lim_{n\to\infty} \mu(A\cap n),$$
% so there is $n$ with $\mu(A\cap n) > \widehat{\phi}_\sigma(A)-\varepsilon$.
% Since $\widehat{\phi}_\sigma \geq \mu $, we obtain  $\widehat{\phi}_\sigma(A\cap n)\geq \widehat{\phi}_\sigma(A)-\varepsilon$. 

(\ref{prop:pathology-Topsoe:stronger-then-Farah})
Because we assumed that $\phi$ is nonzero.

(\ref{prop:pathology-Topsoe:infinite})
The function  
    $$\mu(A)=\begin{cases}
        0&\text{if $\phi(A)<\infty$},\\
        \infty & \text{if $\phi(A)=\infty$}
    \end{cases}$$
is a nonzero measure  which is dominated by $\phi$. Hence, $\widehat{\phi}\neq0$.

(\ref{prop:pathology-Topsoe:lsc})
    Let  $F$ be finite with $\phi(F)>0$.
    Since $\phi$ is subadditive, there is $n\in F$ with $\phi(\{n\})>0$.
Then  the function  
    $$\mu(A)=\begin{cases}        \phi(\{n\})&\text{if $n\in A$},\\
        0&\text{otherwise}
    \end{cases}$$
 is a nonzero $\sigma$-measure which is dominated by $\phi$. Hence $\widehat{\phi}_\sigma\neq 0$.

% The ``in particular'' part follows from the fact that if  $\phi$  is nonzero, then $\phi(\omega)>0$, so using the fact that   $\phi$ is lower semicontinuous, we can find a finite set $F$ with $\phi(F)>0$.

(\ref{prop:properties-of-submeasures:vanishes-on-FIN})
Straightforward.

 (\ref{prop:properties-of-phi-hat:finite-sets})
Since $\widehat{\phi}_\sigma\leq \widehat{\phi}$,  we only need to show 
$\widehat{\phi}(F)\geq \widehat{\phi}_\sigma(F)$ for finite $F$.
There is a measure $\nu$ on $\omega$ such that $\nu\leq\phi$ and $\nu(F)=\widehat{\phi}(F)$.
We define $\mu:\cP(\omega)\to[0,\infty]$ by $\mu(A) = \nu(A\cap F)$.
Then $\mu$ is a~$\sigma$-measure, $\mu\leq \phi$ and $\mu(F)=\widehat{\phi}(F)$, thus we
 obtain $\widehat{\phi}_\sigma(F)\geq \widehat{\phi}(F)$.

(\ref{prop:properties-of-phi-hat:cofinite-set})
Let $F$ be a finite set such that  $\phi(\omega\setminus F)=0$.
For any $A\subseteq\omega$ we have
$\widehat{\phi}(A\cap F)\leq \widehat{\phi}(A)
\leq 
\widehat{\phi}(A\cap F)+\widehat{\phi}(A\setminus F)
=\widehat{\phi}(A\cap F)
$,
so 
$\widehat{\phi}(A\cap F)=\widehat{\phi}(A)$.
The same argument shows that 
$\widehat{\phi}_\sigma(A\cap F)=\widehat{\phi}_\sigma(A)$ for every $A\subseteq\omega$.
Since $A\cap F$ is finite, we can use item~ (\ref{prop:properties-of-phi-hat:finite-sets})
 to obtain 
 $\widehat{\phi}_\sigma(A) = \widehat{\phi}(A)$ for every $A\subseteq\omega$.

(\ref{prop:properties-of-submeasures:measures:item-measure})
Obvious.

(\ref{prop:properties-of-phi-hat:sigma-nopath-implies-lsc})
It follows from item (\ref{prop:properties-of-phi-hat:sigma-hat-is-lsc}).

(\ref{prop:properties-of-phi-hat:measure-not-sigma-meaasure})
It follows from items 
(\ref{prop:properties-of-submeasures:measures:item-measure})--(\ref{prop:properties-of-phi-hat:sigma-nopath-implies-lsc})
and the fact that a lsc measure is a $\sigma$-measure.

(\ref{prop:properties-of-phi-hat:all-sets-for-lsc})
Since $\widehat{\phi}_\sigma\leq \widehat{\phi} = \phi$,  we only need to show 
$\widehat{\phi}_\sigma(A)\geq \phi(A)$ for any $A\subseteq\omega$.
Take any  $M<\widehat{\phi}(A)$.
Since $\phi$ is lsc, we can find a finite set $F\subseteq A$ with $\phi(F)>M$.
Using item  (\ref{prop:properties-of-phi-hat:finite-sets}), 
we obtain
$\phi(F) = \widehat{\phi}(F) = \widehat{\phi}_\sigma(F)$.
Then  $\widehat{\phi}_\sigma(A)\geq \widehat{\phi}_\sigma(F) = \phi(F)>M$.
Hence $\widehat{\phi}_\sigma(A)\geq \widehat{\phi}(A)$.

(\ref{prop:properties-of-phi-hat:all-sets-for-lsc:not-lsc})
It follows from item (\ref{prop:properties-of-phi-hat:all-sets-for-lsc}).

(\ref{prop:properties-of-phi-hat:all-sets-for-lsc:lsc})
It follows from item (\ref{prop:properties-of-phi-hat:sigma-hat-is-lsc}).

 (\ref{prop:properties-of-phi-hat:nonpathological})
The proof in the  case of $\widehat{\phi}_\sigma$ is almost the same as the one for $\widehat{\phi}$, so we provide only the proof for $\widehat{\phi}$.

First, we show that the submeasure $\widehat{\phi}$ is nonpathological.
Since $\widehat{\widehat{\phi}}\leq \widehat{\phi}$,  we only need to show 
$\widehat{\widehat{\phi}}(A)\geq \widehat{\phi}(A)$ for any $A\subseteq\omega$.
Take any  $M<\widehat{\phi}(A)$.
Then there exists a measure $\mu\leq \phi$ such that $\mu(A)>M$.
Since $\widehat{\phi}$ is the supremum of all measures dominated by $\phi$, we obtain $\mu\leq \widehat{\phi}$.
Consequently, 
$\widehat{\widehat{\phi}}(A) \geq \mu(A)>M$.
Since $M<\widehat{\phi}(A)$ was arbitrary, we obtain $\widehat{\widehat{\phi}}(A)\geq \widehat{\phi}(A)$.

Second, we show that $\widehat{\phi}$ is the largest nonpathological submeasure dominated by $\phi$. Take any nonpathological submeasure $\psi\leq \phi$.
We take any $A\subseteq\omega$ and any $M<\psi(A)$. If we show that $\widehat{\phi}(A)>M$, the proof will be finished.
Since $\psi$ is nonpathological, there is a measure $\mu\leq \psi$ with $\mu(A)>M$.
Since $\mu\leq \psi\leq \phi$, we obtain $\widehat{\phi}(A)\geq \mu(A)>M$.
\end{proof}

%%%%%%%%%%%%%%%%%%%%%%%%%%%%%%%%%%%%%%%%
%%%
%%%%%%%%%%%%%%%%%%%%%%%%%%%%%%%%%%%%%%%%

\section{Examples of nonpathological submeasures}
\label{sec:examples-of-nonpathol}

In Section~\ref{sec:Measures}, we defined the submeasures $\delta_\I$ and $\delta_\I^\infty$ for any ideal $\I$ on $X$. In the case of two simple ideals on $\omega$, we will use the following notations.
\begin{itemize}
    \item 
For $\I = \{\emptyset\}$, we write
    $\delta = \delta_\I$ and $\delta^\infty = \delta_\I^\infty$.
    \item 
For $\I = \{A\subseteq \omega:n\in A\}$ with  a fixed $n\in \omega$, we write 
    $\delta_n = \delta_\I$ and $\delta_n^\infty = \delta_\I^\infty$.
\end{itemize}

\begin{proposition} 
\label{prop:properties-of-delta-submeasures}
Let  $\I$ be an ideal on $\omega$ which contains all finite subsets of $\omega$.
\begin{enumerate}

\item \label{prop:properties-of-delta-submeasures:delta}\label{prop:properties-of-delta-submeasures:delta-n}
\begin{enumerate}
    \item $\delta$ is  a $\sigma$-nonpathological and lower semicontinuous submeasures which is not a measure.
    \item $\delta^\infty$, $\delta_n$ and $\delta_n^\infty$ are  $\sigma$-measures (hence lower semicontinuous and $\sigma$-nonpathological).\label{prop:properties-of-delta-submeasures:delta-n-sigma-measure}\label{prop:properties-of-delta-submeasures:delta-n-hat}

\end{enumerate}

\item \label{prop:properties-of-delta-submeasures:delta-ideal}
\begin{enumerate}
\item $\delta_\I$ is not  lower semicontinuous nor a $\sigma$-measure.\label{prop:properties-of-delta-submeasures:delta-ideal:submeasure}
\item $\delta_\I$ is a measure $\iff$ $\I$ is a maximal ideal.\label{prop:properties-of-delta-submeasures:delta-ideal:measure}
\item $0=\widehat{(\delta_\I)}_\sigma \neq \widehat{\delta_\I}=\delta_\I$.\label{prop:properties-of-delta-submeasures:delta-ideal:hat}
\end{enumerate}

\item\label{prop:properties-of-delta-submeasures:delta-plus-delta-I} 
\begin{enumerate}
\item $\delta+\delta_\I$ is a non-lsc submeasure.\label{prop:properties-of-delta-submeasures:delta-plus-delta-I:not-lsc} 
\item 
$\delta = \widehat{(\delta+\delta_\I)}_\sigma \neq \widehat{\delta+\delta_\I}=\delta+\delta_\I$.\label{prop:properties-of-delta-submeasures:delta-plus-delta-I:hat} 
\end{enumerate}

\item\label{prop:properties-of-delta-submeasures:delta-ideal-infinity} 
\begin{enumerate}
\item  $\delta_\I^\infty$ is a measure which is not  lower semicontinuous nor a $\sigma$-measure.\label{prop:properties-of-delta-submeasures:delta-ideal-infinity:submeasure}

\item $0=\widehat{(\delta_\I^\infty)}_\sigma \neq \widehat{\delta_\I^\infty}=\delta_\I^\infty$.\label{prop:properties-of-delta-submeasures:delta-ideal-infinity:hat}
\end{enumerate}

\item \label{prop:properties-of-delta-submeasures:delta+delta+measure}

\begin{enumerate}
\item  $\delta_\I^\infty+\delta_n$ is a measure which is not  lower semicontinuous nor a $\sigma$-measure.\label{prop:properties-of-delta-submeasures:delta+delta+measure:measure}

\item $0 \neq \widehat{(\delta_\I^\infty+\delta_n)}_\sigma \neq \widehat{\delta_\I^\infty+\delta_n}=\delta_\I^\infty+\delta_n$.\label{prop:properties-of-delta-submeasures:delta+delta+measure:hat}

\end{enumerate}

\item \label{prop:properties-of-delta-submeasures:delta+delta+not-measure}

\begin{enumerate}
\item  $\delta_\I+\delta_n$ is a submeasure which is not  lower semicontinuous.\label{prop:properties-of-delta-submeasures:delta+delta+not-measure:submeasure}

\item  $\delta_\I+\delta_n$ is a  measure $\iff$ $\I$ is maximal.\label{prop:properties-of-delta-submeasures:delta+delta+not-measure:measure}

\item $0 \neq \widehat{(\delta_\I+\delta_n)}_\sigma \neq \widehat{\delta_\I+\delta_n}=\delta_\I+\delta_n$.\label{prop:properties-of-delta-submeasures:delta+delta+not-measure:hat}

\end{enumerate}

\end{enumerate}
\end{proposition}

\begin{proof}
(\ref{prop:properties-of-delta-submeasures:delta}),
(\ref{prop:properties-of-delta-submeasures:delta-ideal:submeasure}),
(\ref{prop:properties-of-delta-submeasures:delta-ideal:measure}) 
and 
(\ref{prop:properties-of-delta-submeasures:delta-ideal-infinity:submeasure})
Straightforward. 

(\ref{prop:properties-of-delta-submeasures:delta-ideal:hat})
The equality 
$\widehat{(\delta_\I)}_\sigma = 0$
follows from Proposition~\ref{prop:properties-of-submeasures}(\ref{prop:properties-of-submeasures:vanishes-on-FIN}).
Since $\widehat{\delta_\I}\leq \delta_\I$, we only need to show the converse inequality.
If $A\in \I$, then $\delta_\I(A)=0\leq \widehat{\delta_\I}(A)$, so we can assume that $A\notin \I$.
Then we can find a maximal ideal $\J$ such that $ A\notin \J$ and $\I\subseteq \J$. Since  $\delta_\J$ is a measure, $\delta_\J\leq \delta_\I$ and $\widehat{\delta_\I}(A)\geq \delta_\J(A)=1\geq \delta_\I(A)$, the proof is finished.

(\ref{prop:properties-of-delta-submeasures:delta-plus-delta-I:not-lsc})
Follows from 
$(\delta+\delta_\I)(\omega)=2 $
and
$(\delta+\delta_\I)(F)=1$ for every finite $F\subseteq \omega$.

(\ref{prop:properties-of-delta-submeasures:delta-plus-delta-I:hat})
Since $\delta$ and $\delta_\I$ are nonpathological, $\delta+\delta_\I$ is nonpathological as well.  
To show the first equality, we first observe that $\widehat{(\delta+\delta_\I)}_\sigma
\geq 
\widehat{\delta}_\sigma = \delta
$.
On the other hand, if $\mu$ is a measure such that $\mu\leq \delta+\delta_\I$, then $\mu(F)\leq (\delta+\delta_\I)(F) = 1$, so by using lsc of $\widehat{(\delta+\delta_\I)}_\sigma$ we obtain that  
$\widehat{(\delta+\delta_\I)}_\sigma(A)\leq 1 =\delta(A)$
for every nonempty $A$.

(\ref{prop:properties-of-delta-submeasures:delta-ideal-infinity:hat})
The equality 
$\widehat{(\delta_\I^\infty)}_\sigma = 0$
follows from Proposition~\ref{prop:properties-of-submeasures}(\ref{prop:properties-of-submeasures:vanishes-on-FIN}),
whereas the equality 
$\widehat{(\delta_\I^\infty)} = \delta_\I^\infty$
follows from Proposition~\ref{prop:properties-of-submeasures}(\ref{prop:properties-of-submeasures:measures:item-measure}).

(\ref{prop:properties-of-delta-submeasures:delta+delta+measure:measure})
It follows from items (\ref{prop:properties-of-delta-submeasures:delta-n-sigma-measure}) and (\ref{prop:properties-of-delta-submeasures:delta-ideal-infinity:submeasure}).

(\ref{prop:properties-of-delta-submeasures:delta+delta+measure:hat})
The first inequality follows from the fact that $\delta_n$ is a nonzero $\sigma$-measure.
The  equality follows from the fact that $\delta_\I^\infty+\delta_n$ is a measure. 
The second inequality follows from the fact that 
$\widehat{(\delta_\I^\infty+\delta_n)}_\sigma$ is lower semicontinuous (by Proposition~\ref{prop:properties-of-phi-hat}(\ref{prop:properties-of-phi-hat:sigma-hat-is-lsc})), whereas $\widehat{\delta_\I^\infty+\delta_n}= \delta_\I^\infty+\delta_n$ is not lower semicontinuous.

(\ref{prop:properties-of-delta-submeasures:delta+delta+not-measure:submeasure})
It follows from items (\ref{prop:properties-of-delta-submeasures:delta-n-sigma-measure}) and (\ref{prop:properties-of-delta-submeasures:delta-ideal:submeasure}).

(\ref{prop:properties-of-delta-submeasures:delta+delta+not-measure:measure})
It follows from items (\ref{prop:properties-of-delta-submeasures:delta-n-sigma-measure}) and (\ref{prop:properties-of-delta-submeasures:delta-ideal:measure}).

(\ref{prop:properties-of-delta-submeasures:delta+delta+not-measure:hat})
The first inequality follows from the fact that $\delta_n$ is a nonzero $\sigma$-measure.
The  equality follows from the fact that $\widehat{\delta_\I}=\delta_\I$ and $\delta_n$ is a measure. 
The second inequality follows from the fact that 
$\widehat{(\delta_\I+\delta_n)}_\sigma$ is lower semicontinuous (by Proposition~\ref{prop:properties-of-phi-hat}(\ref{prop:properties-of-phi-hat:sigma-hat-is-lsc})), whereas $\widehat{\delta_\I+\delta_n}= \delta_\I+\delta_n$ is not lower semicontinuous. 
\end{proof}

\begin{proposition}
\label{prop:limsup-of-nonpathol} 
If $\mu_n$ is a measure on $X$ for every $n\in \omega$, then the submeasure 
$$\phi(A) = \limsup_{n\to\infty}\mu_n(A)$$
is nonpathological.
\end{proposition}

\begin{proof}
Take any $A\subseteq X$. Then there exists an increasing sequence $k_0<k_1<\dots$ such that $\phi(A) = \lim_{n\to\infty}\mu_{k_n}(A)$.
Let $\cU$ be an ultrafilter on $\omega$ such that $\{k_n:n\in \omega\}\in \cU$.
Then the measure 
$$\nu(A) = \lim_{n\in \cU}\mu_n(A)$$
is dominated by $\phi$ and $\mu(A)=\phi(A)$. Hence $\widehat{\phi}(A)=\phi(A)$.
\end{proof}

For a function $f:\omega\to[0,\infty)$ such that 
$f(0)\neq 0$, 
$$ \sum_{i\in \omega} f(i) =\infty
\text{\ \ and \ \ } 
\lim_{n\to\infty} \frac{ f(n)}{\sum_{i\in n} f(i)}=0,$$
and for a fixed  $n\in \omega\setminus\{0\}$,  we define \label{def:Erdos-Ulam-submeasure}
a $\sigma$-measure
 $\phi_{f,n}$ on $\omega$ by 
$$
\phi_{f,n}(A)=
\frac{\sum_{i\in A\cap n} f(i)}{\sum_{i\in n} f(i)},$$
 and
 two submeasures
 $\phi_f, \overline{\phi}_f$  on $\omega$ by 
$$
\phi_f(A) = \sup\left\{\phi_{f,n}(A):n\in \omega\right\}
\text{\ \ \ and \ \ \ }
\overline{\phi}_f(A) = \limsup_{n\to\infty}\phi_{f,n}(A).
$$
For a constant function $f=1$ we obtain the \emph{asymptotic density}         
        $$\overline{d}(A) = \limsup_{n\to\infty}\frac{|A\cap n|}{n}.$$

\begin{proposition}
\label{prop:properties-of-asymptotic-densities}
The submeasure $\phi_f$ is $\sigma$-nonpathological and lower semicontinuous, whereas 
 $\overline{\phi}_f$ is nonpathological, but it is not lower semicontinuous.
\end{proposition}

\begin{proof}
The nonpathology of $\overline{\phi}_f$ follows from Proposition~\ref{prop:limsup-of-nonpathol} and the rest of properties are straightforward. 
\end{proof}

%%%%%%%%%%%%%%%%%%%%%%%%%%%%%%%%%%%%%%%%
%%%
%%%%%%%%%%%%%%%%%%%%%%%%%%%%%%%%%%%%%%%%

\section{Examples of pathological submeasures}
\label{sec:examples-of-pathol}

We start with a well known example of a pathological submeasures on a three-element set. This submeasure can be considered as a  blueprint for an example from Proposition~\ref{prop:pathological-submeasure-tau}.

\begin{proposition}[Folklore]
\label{prop:pathological-submeasure-on-3}
The submeasure  $\tau:\{0,1,2\}\to[0,2]$ given by 
    $$\tau(A)=
    \begin{cases}
    0&\text{if $|A|=0$,}\\    
    1&\text{if $1\leq |A|\leq 2$,}\\
    2&\text{if $|A|=3$}
    \end{cases}$$
is pathological and   $0\neq \widehat{\tau}_\sigma=  \widehat{\tau}\neq\tau$.
Moreover, 
    $$\widehat{\tau}(A)=
    \begin{cases}
    0&\text{if $\tau(A)=0$,}\\    
    1&\text{if $\tau(A)=1$,}\\
    3/2&\text{if $\tau(A)=2$.}
    \end{cases}$$
\end{proposition}

\begin{proof}
Using Proposition~\ref{prop:properties-of-phi-hat}(\ref{prop:properties-of-phi-hat:finite-sets}), we obtain  $\widehat{\tau}_\sigma=  \widehat{\tau}$.
    Since $\delta_{0}$ is a nonzero $\sigma$-measure (Proposition~\ref{prop:properties-of-delta-submeasures}(\ref{prop:properties-of-delta-submeasures:delta-n-sigma-measure}))  and $\delta_0\leq \tau$, we obtain $\widehat{\tau}_\sigma\neq 0$.
Finally we  show that  $\widehat{\tau}\neq \tau$ and $\widehat{\tau}(\{0,1,2\})=3/2$.

Let $\mu\leq \tau$ be a measure on $\{0,1,2\}$.
Then  
$\mu(\{0\}) +\ \mu(\{1\})\leq \tau(\{0,1\})=1$,
$\mu(\{0\}) +\ \mu(\{2\})\leq \tau(\{0,1\})=1$
and
$\mu(\{1\}) +\ \mu(\{2\})\leq \tau(\{0,1\})=1$,
hence we obtain 
$2\mu(\{0,1,2\}) \leq 3$.
Thus $\mu(\{0,1,2\})\leq 3/2$, so $\widehat{\tau}\neq \tau$ and $\widehat{\tau}(\{0,1,2\})\leq 3/2$.

Let $\mu$ be a measure on $\{0,1,2\}$ such that $\mu(\{0\})=\mu(\{1\})=\mu(\{2\}) =1/2$. Then $\mu \leq \tau$ and $\mu(\{1,2,3\})=3/2$, so $\widehat{\tau}(\{0,1,2\})\geq 3/2$.
\end{proof}

\begin{proposition}
\label{prop:tau-3-infty}
\label{prop:pathological-submeasure-tau}
Let $\tau$ be a submeasure from Proposition~\ref{prop:pathological-submeasure-on-3}.
Let $A_0,A_1,A_2\subseteq \omega$ be infinite and pairwise disjoint. 

\begin{enumerate}
    \item 
The submeasure $\tau_{3}:\cP(\omega)\to[0,2]$ given by
$$\tau_3(A) = \tau(A\cap \{0,1,2\})$$
is pathological, lower semicontinuous 
and $0 \neq \widehat{(\tau_3)}_\sigma=\widehat{\tau_3}\neq\tau_3$.
Moreover, \label{prop:pathological-submeasure-tau:formula-for-hat}
    $$\widehat{\tau_3}(A)=
    \begin{cases}
    0&\text{if $\tau_3(A)=0$,}\\    
    1&\text{if $\tau_3(A)=1$,}\\
    3/2&\text{if $\tau_3(A)=2$.}
    \end{cases}$$\item 
The submeasure $\tau_3^\infty:\cP(\omega)\to[0,2]$ given by
    $$\tau_3^\infty(A)=
    \begin{cases}
    0&\text{if $|\{i<3:|A\cap A_i|=\omega\}|=0$,}\\    
    1&\text{if $1\leq |\{i<3:|A\cap A_i|=\omega\}|\leq 2$,}\\
    2&\text{if $|\{i<3:|A\cap A_i|=\omega\}|=3$}
    \end{cases}$$
is pathological, not lower semicontinuous  and 
$0 = \widehat{(\tau_3^\infty)}_\sigma\neq\widehat{\tau_3^\infty}\neq \tau_3^\infty$.\label{prop:tau-3-infty:formula-for-hat}
Moreover, 
    $$\widehat{\tau_3^\infty}(A)=
    \begin{cases}
    0&\text{if $\tau^\infty_3(A)=0$,}\\    
    1&\text{if $\tau^\infty_3(A)=1$,}\\
    3/2&\text{if $\tau^\infty_3(A)=2$.}
    \end{cases}$$
\end{enumerate}
\end{proposition}

\begin{proof}
Obviously  $\tau_3$ is lower semicontinuous while $\tau_3^\infty$ is not lower semicontinuous because it vanishes on finite sets (see Proposition~\ref{prop:properties-of-submeasures}(\ref{prop:properties-of-submeasures:vanishes-on-FIN})).

To show that $\widehat{\tau_3}\neq\tau_3$, $\widehat{\tau_3^\infty}\neq\tau_3^\infty$ and obtain the values of $\widehat{\tau_3}$, it is enough to 
imitate  the proof of Proposition~\ref{prop:pathological-submeasure-on-3}, so we omit the details.

Since $\tau_3^\infty$ vanishes on finite sets, we obtain 
$\widehat{(\tau_3^\infty)}_\sigma=0$ by Proposition~\ref{prop:properties-of-submeasures}(\ref{prop:properties-of-submeasures:vanishes-on-FIN}), whereas by Proposition~\ref{prop:pathology-Topsoe}(\ref{prop:pathology-Topsoe:lsc}), we obtain that 
$\widehat{(\tau_3)}_\sigma\neq 0$.

Since $\tau_3(\omega\setminus \{0,1,2\})=0$, we can use Proposition~\ref{prop:properties-of-phi-hat}(\ref{prop:properties-of-phi-hat:cofinite-set}) to obtain $\widehat{(\tau_3)}_\sigma=\widehat{\tau_3}$.

Finally we calculate the value $\widehat{\tau^\infty_3}(A)$ for every set $A$.

If $\tau^\infty_3(A)=0$, then $\widehat{\tau^\infty_3}(A)\leq \tau^\infty_3(A)=0$, so we are done.

If $\tau^\infty_3(A)=1$, then there is $i<3$ such that $A\cap A_i$ is infinite.
Let $\I$ be a maximal ideal on $\omega$ which contains all finite sets and  $A\cap A_i\notin \I$.
Then $\delta_\I$ is a nonzero measure (by Proposition~\ref{prop:properties-of-delta-submeasures}(\ref{prop:properties-of-delta-submeasures:delta-ideal:measure})).
We show that $\delta_\I\leq \tau_3^\infty$.
If $B\in \I$, then $\delta_\I(B)=0\leq \tau_3^\infty(B)$. Assume that $B\notin \I$.
Since $\I$ is maximal, $B\cap A_i\notin\I$. Consequently, $A_i\cap B$ is infinite, so $\tau_3^\infty(B)\geq 1=\delta_\I(B)$.
Thus $1\leq \widehat{\tau_3^\infty}(A)\leq \tau_3^\infty(A)=1$.

If $\tau^\infty_3(A)=2$, we first observe that the inequality 
$\widehat{\tau^\infty_3}(A)\leq 3/2$ can be shown as in the proof of Proposition~\ref{prop:pathological-submeasure-on-3}.
To show the reverse inequality, we take three maximal ideals $\I_0,\I_1$ and $\I_2$ on $\omega$ which contain all finite sets and  $A\cap A_i\notin \I_i$ for $i=0,1,2$. 
Then we define a measure 
$\mu=(\delta_{\I_0}+\delta_{\I_1}+\delta_{\I_2})/2$ 
and repeating the argument from the previous paragraph, we show that $\mu\leq \tau_3^\infty$. Since $\mu(A)=3/2$, the proof is finished.
\end{proof}

The following theorem shows that there are $\varepsilon$-pathological submeasures on finite sets for arbitrarily small $\varepsilon$. Then using this theorem we can  construct three types of  pathological submeasures on $\omega$ (see Proposition~\ref{prop:pathological-lsc-submeasures}).

\begin{theorem}[{\cite[Theorem~1]{MR0412369}}]
\label{thm:herer-christensen:pathology-on-finite-set}
    Let $\varepsilon > 0$ be an arbitrary positive number. 
    There exist a finite set $X$ and a submeasure $\phi$ on $X$ such that $\phi(X)=1$ and 
    $\mu(X) \leq \varepsilon$
    for  any measure $\mu\leq \phi$. 
\end{theorem}

\begin{proposition}
\label{prop:pathological-lsc-submeasures}
Using Theorem~\ref{thm:herer-christensen:pathology-on-finite-set}, we construct 
 a partition  $\{I_n:n<\omega\}$  of $\omega$ into finite intervals and submeasures $\phi_n$ on $I_n$ such that $\phi_n(I_n)=1$ and 
 $\mu(I_n)\leq 1/2^{n+1}$ for every measure $\mu_n\leq \phi_n$ and each $n<\omega$.
    \begin{enumerate}
    \item 
The  lower semicontinuous submeasure  $\psi_0$ on $\omega$ given by 
$$\psi_0(A)=\phi_0(A\cap I_0)$$
is pathological and     
    $0\neq \widehat{(\psi_0)}_\sigma  = \widehat{\psi_0} \neq \psi_0$.

        \item 
The  lower semicontinuous submeasure  $\psi_1$ on $\omega$ given by 
$$\psi_1(A)  =  \sup\{\phi_n(A\cap I_n):n<\omega\}$$
is pathological.

    \item 
The  lower semicontinuous submeasure  $\psi_2$ on $\omega$ given by 
$$\psi_2(A) =  \sum_{n<\omega}\phi_n(A\cap I_n)$$ 
is pathological,  
    $0\neq \widehat{(\psi_2)}_\sigma  \neq  \widehat{\psi_2} \neq \psi_2$ 
    and $\widehat{(\psi_2)}_\sigma(\omega)<\infty=\widehat{\psi_2}(\omega)$.

            \item 
The  lower semicontinuous submeasure  $\psi_3$ on $\omega$ given by 
$$\psi_3(A)  =  \sup\{(n+1)\phi_n(A\cap I_n):n<\omega\}$$
is pathological,
    $0\neq \widehat{(\psi_3)}_\sigma  \neq  \widehat{\psi_3} \neq \psi_3$, 
$\widehat{(\psi_3)}_\sigma(\omega)<\infty=\widehat{\psi_3}(\omega)$
and 
$\widehat{\psi_3}$ is \emph{not} lower semicontinuous.\label{prop:pathological-lsc-submeasures:item-non-lsc-hat} 

\end{enumerate}
\end{proposition}

\begin{proof}
Since $\psi_i(I_0)=1>0$ for each $i$, we obtain  $\widehat{(\psi_i)}_\sigma\neq 0$ by Proposition~\ref{prop:pathology-Topsoe}(\ref{prop:pathology-Topsoe:lsc}).

We show that $\psi_i$ is pathological for each $i$.
We fix  $i<4$ and take any measure $\mu\leq \psi_i$.
Since $\mu\restriction \cP(I_0)$ is a measure on $I_0$  which is dominated by $\phi_0$, we obtain that $\mu(I_0)\leq 1/2 < 1 = \psi_i(I_0) $.
Thus $\widehat{\psi_i}\neq \psi_i$.

We show $\widehat{(\psi_0)}_\sigma  = \widehat{\psi_0}$.
It is enough to observe that for any  measure $\mu\leq \psi_0$ we have $\mu(\omega\setminus I_0)\leq \psi_0(\omega\setminus I_0)=0$. Thus, $\mu$ takes nonzero values only for subsets of a finite set $I_0$, and consequently $\mu$ is a $\sigma$-measure.

We show  $\widehat{(\psi_i)}_\sigma  \neq  \widehat{\psi_i}$ and $\widehat{\psi_i}(\omega)=\infty$ for $i\in \{2,3\}$.
Since $\phi_n(I_n)=1$ for every $n$, we obtain that $\psi_i(\omega) =\infty$.
Let $\mu:\cP(\omega)\to[0,\infty]$ be given by 
$$\mu(A) = \begin{cases}
    \infty&\text{ if $\psi_i(A)=\infty$,}\\
    0&\text{otherwise.}
\end{cases}$$ 
Since  $\mu$ is a measure and $\mu\leq \psi_i$, we obtain $\widehat{\psi_i}(\omega)\geq \mu(\omega)=\infty$.
Once we show that $\widehat{(\psi_i)}_\sigma(\omega)<\infty$, the proof of this case will be finished.
Let $\mu\leq\psi_i$ be a $\sigma$-measure.
Since $\mu\restriction\cP(I_k)$ is a measure on $I_k$ dominated by $(k+1)\phi_k$, we obtain that 
$$\mu\left(\bigcup_{k<n}I_k\right) 
= 
\sum_{k<n}\mu\left(I_k\right)
\leq 
\sum_{k<n}\frac{k+1}{2^{k+1}}
< 
\sum_{k<\omega}\frac{k+1}{2^{k+1}} = 2 <\infty
$$
for every $n\in \omega$.
Since $\widehat{(\psi_i)}_\sigma$ is lower semicontinuous (by Proposition~\ref{prop:properties-of-phi-hat}(\ref{prop:properties-of-phi-hat:sigma-hat-is-lsc})), we have 
$$\widehat{(\psi_i)}_\sigma(\omega) 
= 
\lim_{n\to\infty} \widehat{(\psi_i)}_\sigma\left(\bigcup_{k<n}I_k\right)
\leq 2 <\infty.$$

Finally, we show that $\widehat{\psi_3}$ is not lower semicontinuous. In the previous paragraph, we showed that $\widehat{\psi_3}(\omega)=\infty$. 
By Proposition~\ref{prop:properties-of-phi-hat}(\ref{prop:properties-of-phi-hat:finite-sets}), we know that $\widehat{\psi_3}(F)=\widehat{(\psi_3)}_\sigma(F)$ for every finite set $F$, so  
$$\widehat{\psi_3}(F) = \widehat{(\psi_3)}_\sigma(F) \leq 2,$$
and consequently 
$$\lim_{n\to\infty} \widehat{\psi_3}(n) \leq 2 < \infty =\widehat{\psi_3}(\omega).$$
\end{proof}

\begin{theorem}[{\cite[Theorem~2]{MR0412369} (see also \cite[Example~3]{MR0419712})}]
\label{thm:herer-christensen:phi-hat-equals-zero}
    There exists a nonzero submeasure $\chi$ on $\omega$ such that $\widehat{\chi}=0$, $\chi(\omega)=1$ 
    and $\chi(F)=0$ for every finite set $F\subseteq\omega$.
\end{theorem}

%%%%%%%%%%%%%%%%%%%%%%%%%%%%%%%%%%%%%%%%
%%%
%%%%%%%%%%%%%%%%%%%%%%%%%%%%%%%%%%%%%%%%

\section{Degree of pathology of submeasures}
\label{sec:degree-of-pathol}

As in the case of  definitions of pathological submeasure, there is also a mess with definitions of degrees of pathology of submeasures.  In this section we try to clean it up. The results we obtained are summarized in Table~\ref{tab:P-like-degrees}.

\begin{definition}
\ 
\begin{enumerate}
    \item 
The \emph{degree of pathology} (\cite[p.~21]{MR1711328}) of a submeasure $\phi$ on $X$  is given by 
$$P(\phi)=\sup \left\{ \frac{\phi(A)}{\widehat\phi(A)}: A \subseteq X \right\}$$
with the convention that $\infty/\infty=0/0=1$ and $a/0=\infty/a=\infty$ for a positive $a\in \R$.

\item 
We also define a ``$\sigma$-version'' of the degree of pathology by 
$$P_\sigma(\phi)=\sup \left\{ \frac{\phi(A)}{\widehat{\phi}_\sigma(A)}: A \subseteq X \right\}.$$

\item 
In \cite[Section~3.1]{MR4797308} (see also \cite[p.~5]{martinez2022pathology} or \cite[p.~3]{martinez2022fsigma}), the authors introduced a ``${\fin}$-variant'' of the degree of pathology by 
$$P_{\fin}(\phi)=\sup \left\{ \frac{\phi(F)}{\widehat\phi(F)}: \text{$F$ is a finite subset of $X$} \right\}.$$

\item Since $\widehat{\phi}(F)=\widehat{\phi}_\sigma(F)$ for every finite set $F\subseteq\omega$ 
(by Proposition~\ref{prop:properties-of-phi-hat}(\ref{prop:properties-of-phi-hat:finite-sets})), a ``$\sigma$-${\fin}$-variation'' of the degree of pathology coincides with $P_{\fin}$.
\end{enumerate}
\end{definition}

The following proposition summarize basic properties of these three versions of the degree of pathology. 

\begin{proposition}
\label{prop:basic-prop-of-P-phi}    
Let $\phi$ be a submeasure on $X$.
\begin{enumerate}

\item $1\leq P_{\fin}(\phi)\leq P(\phi)\leq P_\sigma(\phi)\leq \infty$.\label{prop:basic-prop-of-P-phi:inequalities}    

\item \label{prop:basic-prop-of-P-phi:nonpathology:item}
\begin{enumerate}
\item $\phi$ is nonpathological $\iff$ $P(\phi)=1$.\label{prop:basic-prop-of-P-phi:nonpathology}    
\item  $\phi$ is $\sigma$-nonpathological $\iff$ $P_\sigma(\phi)=1$.\label{prop:basic-prop-of-P-phi:sigma-nonpathology}    
\end{enumerate}

\item If $\phi$ is  a lower semicontinuous submeasure on $\omega$, then\label{prop:basic-prop-of-P-phi:nonpathology-for-lsc}
\begin{enumerate}
\item $P_{\fin}(\phi)=P(\phi) =P_\sigma(\phi)$,

\item
$\phi$   is nonpathological $\iff$ $P_{\fin}(\phi)=1.$

\end{enumerate}

\end{enumerate}
\end{proposition}

\begin{proof}
Straightforward with the aid of Proposition~\ref{prop:properties-of-submeasures}.
\end{proof}

In Table~\ref{tab:P-like-degrees}, 
we present  examples  which show that almost all  configurations of ``$<$'' and ``$=$'' signs in item~(\ref{prop:basic-prop-of-P-phi:inequalities}) of Proposition~\ref{prop:basic-prop-of-P-phi} are possible. 
Below we provide some proofs for the  examples presented in Table~\ref{tab:P-like-degrees}.

% \begin{table}
%     \centering
%     \begin{tabular}{|c|c|c|c||c|}
% \hline
%  &&&& Example  \\ \hline \hline
%           $<$ & $=$ &   $=$ &   $<$ &   \\ \hline
%           $<$ & $=$ &   $=$ &   $<$ &   \\ \hline
%           $<$ & $=$ &   $=$ &   $<$ &   \\ \hline
%           $<$ & $=$ &   $=$ &   $<$ &   \\ \hline
%           $<$ & $=$ &   $=$ &   $<$ &   \\ \hline
%           $<$ & $=$ &   $=$ &   $<$ &   \\ \hline
%           $<$ & $=$ &   $=$ &   $<$ &   \\ \hline
%           $<$ & $=$ &   $=$ &   $<$ &   \\ \hline
%           $<$ & $=$ &   $=$ &   $<$ &   \\ \hline
%     \end{tabular}
%     \caption{All configurations of ``$<$'' and ``$=$'' signs in item~(\ref{prop:basic-prop-of-P-phi:inequalities}) of Proposition~\ref{prop:basic-prop-of-P-phi} are possible.}
%     \label{tab:P-like-degrees}
% \end{table}

\begin{table}
    \centering
    \begin{tabular}{|c|c|c|c|c|c|c|c|c||l|}
\hline
1 & $\leq$ & $P_{\fin}(\phi)$ & $\leq$ & $P(\phi)$ & $\leq$ & $P_{\sigma}(\phi)$ & $\leq$ & $\infty$ & Example \\ \hline \hline
   1      & $=$ & 1  & $=$ & 1 & $=$ & 1 & $<$ & $\infty$ & Every $\sigma$-measure.\\ \hline
   1      & $=$ & 1  & $=$ & 1 & $<$ & $\infty$ & $=$ & $\infty$ & $\delta^\infty_{\fin}$ (Prop.~\ref{prop:properties-of-delta-submeasures}(\ref{prop:properties-of-delta-submeasures:delta-ideal-infinity:hat}))\\ \hline
   1      & $=$ & 1  & $<$ & $\infty$ & $=$ & $\infty$ & $=$ & $\infty$ & $\chi$ (Thm.~\ref{thm:herer-christensen:phi-hat-equals-zero}) \\ \hline
   1      & $<$ & $\infty$  & $=$ & $\infty$ & $=$ & $\infty$ & $=$ & $\infty$ & $\psi_1$, $\psi_2$ (Prop.~\ref{prop:pathological-lsc-submeasures:P})\\ \hline \hline
   1      & $=$ & 1  & $=$ & 1 & $<$ & 2  & $<$ & $\infty$ & $\delta+\delta_{\fin}$ (Prop.~\ref{prop:delta+delta-FIN:P})\\ \hline
   1      & $=$ & 1 & $<$ &   & $=$ & 2 & $<$ & $\infty$ &  \textbf{?} \hfill  $ \delta+\chi$ (Prop.~\ref{prop:delta+chi:P}) \\ \hline
   1      & $=$ & 1 & $<$ & 4/3 & $<$ & $\infty$ & $=$ & $\infty$& $\tau_3^\infty$ (Prop.~\ref{prop:tau-inf:P})) \\ \hline
   1      & $<$ & 4/3 & $=$ & 4/3 & $=$ & 4/3 & $<$ & $\infty$& $\tau_3$ (Prop.~\ref{prop:pathological-submeasure-tau}(\ref{prop:pathological-submeasure-tau:formula-for-hat}))  \\ \hline
   1      & $<$ & 4/3 & $=$ & 4/3 & $<$ & $\infty$ & $=$ & $\infty$ & $\tau_3\oplus \tau_3^\infty$ (Prop.~\ref{prop:oplus-of-submeasures-P})\\ \hline
   1      & $<$ & 4/3 & $<$ & $\infty$ & $=$ & $\infty$ & $=$ & $\infty$ & $\tau_3\oplus \chi$ (Prop.~\ref{prop:oplus-of-submeasures-P}) \\ \hline \hline
   1      & $=$ & 1 & $<$ & 3/2  & $<$ & 6 & $<$ & $\infty$ & $\eta$ (Prop.~\ref{prop:eta}) \\ \hline
   1      & $<$ & 4/3 & $=$ & 4/3 & $<$ & 2 & $<$ & $\infty$& $(\delta+\delta_{\fin})\oplus\tau_3$ (Prop.~\ref{prop:oplus-of-submeasures-P})\\ \hline
   1      & $<$ & 4/3 & $<$ & & $=$ & 2 & $<$ & $\infty$ & \textbf{?} \hfill \  \ $\tau_3 \oplus (\delta+\chi)$  (Prop.~\ref{prop:oplus-of-submeasures-P}) \\ \hline
   1      & $<$ & 4/3 & $<$ & 3/2 & $<$ & $\infty$ & $=$ & $\infty$ & $\tau_3\oplus \tau_3^\infty\oplus \eta$ (Prop.~\ref{prop:oplus-of-submeasures-P}) \\ \hline \hline 
   1      & $<$ & 4/3 & $<$ & 3/2 & $<$ & 6 & $<$ & $\infty$ & $\eta\oplus\tau_3$ (Prop.~\ref{prop:oplus-of-submeasures-P})\\ \hline         
    \end{tabular}
    \caption{Possible configurations of ``$<$'' and ``$=$'' signs in  Proposition~\ref{prop:basic-prop-of-P-phi}(\ref{prop:basic-prop-of-P-phi:inequalities}).}
    \label{tab:P-like-degrees}
\end{table}

\begin{proposition}
\label{prop:pathological-lsc-submeasures:P}
$P_{\fin}(\psi_1)=P_{\fin}(\psi_2)=\infty$.
\end{proposition}

\begin{proof}
It is enough to notice that   $\psi_i(I_n)=1$ and $\widehat{\psi_i}(I_n)\leq 1/2^{n+1}$ for each $n\in \omega$ and $i=1,2$.
\end{proof}

\begin{proposition}
\label{prop:delta+delta-FIN:P}
$P(\delta+\delta_{\fin})=1$ and $P_\sigma(\delta+\delta_{\fin})=2$.
\end{proposition}

\begin{proof}
Because
$\delta+\delta_{\fin}$ is nonpathological 
(Prop.~\ref{prop:properties-of-delta-submeasures}(\ref{prop:properties-of-delta-submeasures:delta-plus-delta-I})), we obtain $P(\delta+\delta_{\fin})=1$ by Prop.~\ref{prop:basic-prop-of-P-phi}(\ref{prop:basic-prop-of-P-phi:nonpathology}).
The equality $P_\sigma(\delta+\delta_{\fin})=2$
follows from the fact that $\widehat{(\delta+\delta_{\fin})}_\sigma=\delta$ (Prop.~\ref{prop:properties-of-delta-submeasures}(\ref{prop:properties-of-delta-submeasures:delta-plus-delta-I})).
\end{proof}

\begin{proposition}
\label{prop:delta+chi:P}
$P_{\fin}(\delta+\chi)=1$
and 
$P_\sigma(\delta+\chi)=2$
\end{proposition}

\begin{proof}
    To show that $P_{\fin}(\delta+\chi)=1$, it is enough to observe that 
for every finite set $F$, 
$(\delta+\chi)(F)=\delta(F)$, so $\widehat{(\delta+\chi)}(F)=\widehat{\delta}(F)$

To show that $P_\sigma(\delta+\chi)=2$, we first observe that $\widehat{(\delta+\chi)}_\sigma
\geq 
\widehat{\delta}_\sigma = \delta
$.
On the other hand, if we have a measure $\mu$ such that $\mu\leq \delta+\chi$, then $\mu(F)\leq (\delta+\chi)(F) = 1$, so by using lsc of $\widehat{(\delta+\chi)}_\sigma$ we obtain that  
$\widehat{(\delta+\chi)}_\sigma(A)\leq 1 =\delta(A)$
for every nonempty $A$.
\end{proof}

\begin{proposition}
\label{prop:tau-inf:P}
$P_{\fin}(\tau_3^\infty)=1$, 
$P(\tau_3^\infty)=4/3$
and
$P_{\sigma}(\tau_3^\infty)=\infty$.
\end{proposition}

\begin{proof}
    Since $\tau_3^\infty(F)=0$ for every finite set $F$, we obtain 
$P_{\fin}(\tau_3^\infty)=1$.
By Prop.~\ref{prop:tau-3-infty},  $\widehat{\tau_3^\infty}=0$, so  we obtain  
$P_\sigma(\tau_3^\infty)=\infty$.
Lastly, 
$P(\tau_3^\infty)=4/3$
follows from Proposition~\ref{prop:tau-3-infty}(\ref{prop:tau-3-infty:formula-for-hat}).
\end{proof}

For   submeasures  $\phi$ and $\psi$ on $\omega$, we define
two submeasures 
$\phi\oplus_m\psi$ 
and $\phi\oplus_s\psi$ on $X = \{0,1\}\times \omega$ by 
\begin{equation*}
    \begin{split}
        (\phi\oplus_m\psi)(A)
        & =
        \max\{\phi(A),\psi(A)\}
\\
(\phi\oplus_s\psi)(A)
        & =
        \phi(A)+\psi(A),
    \end{split}
\end{equation*}
where 
$\phi(A)$ and $\psi(A)$ mean 
$\phi(\{n\in \omega: (0,n)\in A\})$
and 
$\psi(\{n\in \omega: (1,n)\in A\})$, respectively.

\begin{lemma}
\label{lem:oplus-of-submeasures}
For every  submeasures  $\phi$ and $\psi$ on $\omega$ and every $C\subseteq \{0,1\}\times \omega$, we have
\begin{enumerate}
    \item $\widehat{\phi\oplus_m\psi}(C) \geq \max\{\widehat{\phi}(C),\widehat{\psi}(C)\}$,
    \item $\widehat{\phi\oplus_s\psi}(C) = \widehat{\phi}(C)+\widehat{\psi}(C)$.  
\end{enumerate}
The same properties hold for $\widehat{\phi}_\sigma$ and $\widehat{\psi}_\sigma$.
\end{lemma}

\begin{proof}
    (1) 
    Take any $C\subseteq \{0,1\}\times \omega$.
    Without loss of generality, we can assume $\widehat{\phi(C)}\geq \widehat{\psi(C)}$.
    Then there is a measure $\mu$ on $\omega$ such that $\mu\leq \phi$ and $\mu(C)> M$.  Let $\nu$ be a measure on $\{0,1\}\times\omega$ such that $\nu(\{0\}\times A)=\mu(A)$ and $\nu(\{1\}\times A)=0$ for every set $A\subseteq \omega$.
    Then $\nu\leq \phi\oplus_m\psi$ and $\nu(C) = \mu(C) > M$.
    Hence $\widehat{\phi\oplus_m\psi}(C)> M$.

(2,$\leq$)
Take any measure $\mu\leq \widehat{\phi\oplus_s\psi}$.
Let $\mu_i(A)=\mu(\{(i,n):n\in A\})$ for $A\subseteq \omega$ and $i=0,1$.
Then $\mu_0$ and $\mu_1$ are measures on $\omega$ and $\mu_0\leq \phi$, $\mu_1\leq \psi$.
Hence 
$\mu_0(A)\leq \widehat{\phi}(A)$
and 
$\mu_1(A)\leq \widehat{\psi}(A)$ for every $A\subseteq \omega$.
Then 
$\mu(C) = \mu_0(C)+\mu_1(C) \leq \widehat{\phi}(C)+\widehat{\psi}(C)$.

(2,$\geq$)
Take any $C\subseteq \{0,1\}\times \omega$.
    Take any $M_1< \widehat{\phi}(C)$ and $M_2<\widehat{\psi}(C)$. Then there are measures $\mu_1$ and $\mu_2$  on $\omega$ such that $\mu_1\leq \phi$, $\mu_2\leq \psi$ and $\mu_i(C)> M_i$ for $i=1,2$.  Let $\nu=\mu_1\oplus_s\mu_2$.
    Then $\nu\leq \phi\oplus_s\psi$ and $\nu(C) = \mu_1(C)+\mu_2(C)>M_1+M_2 $.
    Hence $\widehat{\phi\oplus_s\psi}(C)> M_1+M_2$.
\end{proof}

\begin{proposition}
\label{prop:oplus-of-submeasures-P}
For every  submeasures  $\phi$ and $\psi$, we have
\begin{enumerate}
    \item $P(\phi\oplus_m\psi) = \max\{P(\phi),P(\psi)\}$,
    \item $P(\phi\oplus_s\psi) = \max\{P(\phi),P(\psi)\}$.
\end{enumerate}
The same equalities hold for $P_\sigma$ and $P_{\fin}$.
\end{proposition}

\begin{proof}
(1, $\leq$)
It is enough to notice, using Lemma~\ref{lem:oplus-of-submeasures}, that for every set $C$ we have  
\begin{equation*}
    \begin{split}
        \frac{\phi\oplus_m\psi(C)}{\widehat{\phi\oplus_m\psi}(C)} 
    \leq 
    \frac{\max\{\phi(C),\psi(C)\}}{\max\{\widehat{\phi}(C),\widehat{\psi}(C)\}} 
    \leq 
\max\left\{\frac{\phi(C)}{\widehat{\phi}(C)}, 
\frac{\psi(C)}{\widehat{\psi}(C)}\right\} 
\leq \max\{P(\phi),P(\psi)\}.  
\end{split}
\end{equation*}

(1, $\geq$)
Without loss of generality, we can assume that $P(\phi)\geq P(\psi)$.
Take any  $M<P(\phi)$.
Then there is $C\subseteq \{0\}\times \omega$ such that 
$\phi(C)/\widehat{\phi}(C)>M$.
Then $(\phi\oplus_m\psi)(C) = \max\{\phi(C),0\}=\phi(C)$, 
and consequently $\widehat{\phi\oplus_m\psi}(C)=\widehat{\phi}(C)$.
Thus,
$$\frac{(\phi\oplus_m\psi)(C)}{\widehat{\phi\oplus_m\psi}(C)}
=
\frac{\phi(C)}{\widehat{\phi}(C)}>M.$$
Hence, $P(\phi\oplus_m\psi)>M$. 

(2) The inequality $\geq$ can be done as in item (1), and to show the inequality $\leq$ it is enough to notice, using Lemma~\ref{lem:oplus-of-submeasures}, that for every set $C$ we have  
\begin{equation*}
    \begin{split}
        \frac{\phi\oplus_s\psi(C)}{\widehat{\phi\oplus_s\psi}(C)} 
    =
    \frac{\phi(C)+\psi(C)}{\widehat{\phi}(C)+\widehat{\psi}(C)} 
    \leq 
\max\left\{\frac{\phi(C)}{\widehat{\phi}(C)}, 
\frac{\psi(C)}{\widehat{\psi}(C)}\right\} 
\leq \max\{P(\phi),P(\psi)\}.  
\end{split}
\end{equation*}

\end{proof}

\begin{proposition}\label{prop:eta}
Let $A_0,A_1,A_2,A_3\subseteq \omega$ be infinite and pairwise disjoint. 
The submeasure $\eta$ on $\omega$ given by
    $$\eta(A)=
    \begin{cases}
    0&\text{if $A=\emptyset$,}\\    
    1&\text{if $A\neq\emptyset$ and $|\{i<4:|A\cap A_i|=\omega\}|=0$,}\\    
    3&\text{if $1\leq |\{i<4:|A\cap A_i|=\omega\}|\leq 3$,}\\
    6&\text{if $|\{i<4:|A\cap A_i|=\omega\}|=4$}
    \end{cases}$$
is pathological, not lower semicontinuous  and 
$0 \neq \delta=\widehat{(\eta)}_\sigma\neq\widehat{\eta}\neq \eta$.\label{prop:eta:formula-for-hat}
Moreover, 
    $$\widehat{\eta}(A)=
    \begin{cases}
    0&\text{if $\eta(A)=0$,}\\    
    1&\text{if $\eta(A)=1$,}\\
    3&\text{if $\eta(A)=3$,}\\
    4&\text{if $\eta(A)=6$.}
    \end{cases}$$
\end{proposition}
\begin{proof}
Clearly, $\eta(F)\leq 1$ for every finite $F\subseteq\omega$ and $\eta(\omega)=6$, so $\eta$ is not lower semicontinuous. It is also an easy observation that $\widehat{(\eta)}_\sigma=\delta$.

%To show that $\widehat{\tau_3}\neq\tau_3$ and $\widehat{\tau_3^\infty}\neq\tau_3^\infty$, it is enough to 
%imitate  the proof of Proposition~\ref{prop:pathological-submeasure-on-3}, so we omit the details.

We will now calculate the value $\widehat{\eta}(A)$ for every set $A$, which will also prove that $\widehat{(\eta)}_\sigma\neq\widehat{\eta}\neq \eta$.

If $\eta(A)=0$, then $\widehat{\eta}(A)\leq \eta(A)=0$, so we are done.

If $\eta(A)=1$, then for any $n\in A$ take the measure $\delta_n$ in order to  obtain $\delta_n\leq \eta$ and $\delta_n(A)=1$. 

If $\eta(A)=3$, then there is $i<4$ such that $A\cap A_i$ is infinite.
Let $\I$ be a maximal ideal on $\omega$ which contains all finite sets and  $A\cap A_i\notin \I$.
Then $3\delta_\I$ is a nonzero measure (by Proposition~\ref{prop:properties-of-delta-submeasures}(\ref{prop:properties-of-delta-submeasures:delta-ideal:measure})).
We show that $3\delta_\I\leq \eta$.
If $B\in \I$, then $3\delta_\I(B)=0\leq \eta(B)$. Assume that $B\notin \I$.
Since $\I$ is maximal, $B\cap A_i\notin\I$. Consequently $A_i\cap B$ is infinite, so $\eta(B)\geq 3=3\delta_\I(B)$.
Thus $3\leq \widehat{\eta}(A)\leq \eta(A)=3$.

If $\eta(A)=6$, 
to show that $\widehat{\eta}(A)\geq 4$, we take four maximal ideals $\I_0,\I_1,\I_2$ and $\I_3$ on $\omega$  such that $A\cap A_i\notin \I_i$ for $i=0,1,2,3$. 
Then we define a measure 
$\mu=(\delta_{\I_0}+\delta_{\I_1}+\delta_{\I_2}+\delta_{\I_3})$ 
and repeating the argument from the previous paragraph, we show that $\mu\leq \eta$. Since $\mu(A)=4$, we obtain $\widehat{\eta}(A)\geq 4$.
To obtain the reverse inequality, we can proceed analogously to the proof of Proposition~\ref{prop:pathological-submeasure-on-3}. 
\end{proof}

\begin{proposition}
Suppose that $\phi$ is a submeasure on $\omega$ such that 
$$1=P_{\fin}(\phi)<P(\phi)=P_\sigma(\phi)<\infty.$$
Let $\alpha\in \R$ be such that  $\alpha\leq 2$ and $3/2<\alpha<(3/2)\cdot P(\phi)$.
Let $\psi$ be a submeasure on $\omega$ given by
$$\psi(A)=
\begin{cases}
    0&\text{if $A\cap \{0,1,2\}=\emptyset$,}\\
    1&\text{if $1\leq |A\cap \{0,1,2\}|\leq 2$,}\\
    \alpha &\text{if $|A\cap \{0,1,2\}|=3$}.
\end{cases}$$
Then 
$$P_{\fin}(\psi)=P(\psi)=P_\sigma(\psi) = (2/3)\cdot \alpha,$$
and consequently,
$$1<P_{\fin}(\phi\oplus\psi)<P(\phi\oplus\psi)=P_\sigma(\phi\oplus\psi)<\infty.$$
\end{proposition}

\begin{proof}
The degree of pathology of $\psi$ can be calculated the same way as for $\tau_3$ (using an appropriate modification of 
Proposition~\ref{prop:pathological-submeasure-tau}).
Whereas the degree of pathology of $\psi\oplus\psi$ can be easily calculated with the aid of Proposition~\ref{prop:oplus-of-submeasures-P}.
\end{proof}

\begin{question}
    \ 

\begin{enumerate}
    \item What is the value of $P(\delta+\chi)$?
    \item Is there a submeasure $\phi$ with 
$1=P_{\fin}(\phi)<P(\phi)=P_\sigma(\phi)<\infty$?
\end{enumerate}
\end{question}

%%%%%%%%%%%%%%%%%%%%%%%%%%%%%%%%%%%%%%%%
%%%
%%%%%%%%%%%%%%%%%%%%%%%%%%%%%%%%%%%%%%%%

\part{(Non)pathological ideals}
\label{part:ideals}

In this part of the paper, we alter the definition of an ideal. Namely, from this point on  we additionally assume that 
\emph{an ideal $\I$ on $X$ has to contain all finite subsets of $X$}.

%%%%%%%%%%%%%%%%%%%%%%%%%%%%%%%%%%%%%%%%
%%%
%%%%%%%%%%%%%%%%%%%%%%%%%%%%%%%%%%%%%%%%

\section{(Non)pathological ideals}
\label{sec:def-of-path-ideal}

In the literature, the (non)pathology of ideals is only considered in the realm of analytic P-ideals and $F_\sigma$ ideals so far. 
Let us recall well-known characterizations of analytic P-ideals and $F_\sigma$ ideals that are expressed in terms of submeasures and are necessary to phrase  the definition of pathological ideals.
\begin{theorem}[\cite{solecki}, see also {\cite[Theorem~1.2.5]{MR1711328}}]
\label{thm:exh-is-analytic-p}
The following conditions are equivalent.
\begin{enumerate}
\item $\I$ is an analytic P-ideal on $\omega$.
\item There is a lower semicontinuous submeasure $\phi$ on $\omega$ such that
$\lim_{n\to\infty} \phi(\omega\setminus n) >0 $  and  
$$\I=\Exh(\phi)=\left\{A\subseteq\omega: \lim_{n\to\infty} \phi(A\setminus n)=0\right\}.$$
\end{enumerate}
\end{theorem}
\begin{theorem}[{\cite[Lemma~1.2]{mazur}}, see also {\cite[Theorem~1.2.5]{MR1711328}}]
The following conditions are equivalent.
\begin{enumerate}
\item
$\I$ is an $F_\sigma$ ideal on $\omega$.
\item There is  a lower semicontinuous submeasure $\phi$ on $\omega$ such that 
$\phi(n)<\infty$ for each $n\in \omega$, $\phi(\omega)=\infty$  and 
$$\I= {\fin}(\phi)=\left\{A\subseteq \omega: \phi(A)<\infty  \right\}.$$
\end{enumerate}
\end{theorem}

Having the above characterizations we are prepared to present the definitions of (non)pathological  ideals.  

\begin{definition}[{\cite[pp.~25 and 53]{MR1711328}}] 
    An analytic P-ideal ($F_\sigma$ ideal, resp.) $\I$ is \emph{nonpathological} if there exists a nonpathological, lower semicontinuous submeasure $\phi$  such that $\I=\Exh(\phi)$ ($\I={\fin}(\phi)$, resp.). Otherwise, we say that $\I$ is \emph{pathological}.
\end{definition}

By the definition, every submeasure which defines  a pathological ideal has to be pathological. On the other hand, every nonpathological ideal can be  always expressed with the aid of some pathological submeasure (\cite[Proposition~3.4]{MR4797308}). Indeed, take any lsc submeasure $\phi$ with $\I = {\fin}(\phi)$ ($\I=\Exh(\phi)$, resp.), then the submeasure $\psi(A) = \tau(A\cap \{0,1,2\})+\phi(A\setminus \{0,1,2\})$ (where $\tau$ is defined in Proposition~\ref{prop:pathological-submeasure-on-3}) 
is pathological (as $P(\psi)\geq P(\tau)=4/3>1$ by Proposition~\ref{prop:oplus-of-submeasures-P}) and ${\fin}(\psi)={\fin}(\phi)$ ($\Exh(\psi)=\Exh(\phi)$, resp.).   

It is known (\cite[Lemma~1.2.2]{MR1711328}) that $\Exh(\phi)\subseteq {\fin}(\phi)$ for every lsc submeasure $\phi$, however there is no relation between pathology of $\exh(\phi)$ and pathology of ${\fin}(\phi)$ in general. Indeed, first we observe that if $\Exh(\phi)$ is a pathological ideal, then  $\psi=\min\{1,\phi\}$ is  a lsc submeasure with $\exh(\phi)=\exh(\psi)$ and ${\fin}(\psi)=\cP(\omega)$, so $\Exh(\psi)$ is pathological whereas $\fin(\psi)$ is not even an ideal. 
If we now take $\zeta(A)=\sup\{\psi(\{n\in\N: 2n\in A\}),|A\cap (2\N+1)| \}$ then $\fin(\zeta)=\cP(\N)\oplus\fin$ would be a proper nonpathological ideal while $\exh(\zeta)=\exh(\psi)\oplus \fin$ would be pathological.
On the other hand, if we take any pathological ideal ${\fin}(\phi)$ and put $\psi=\sup\{\phi,\delta\}$ then  ${\fin}(\psi)={\fin}(\phi)$ would still be pathological while $\exh(\psi)=\fin$ would be nonpathological.

Let us  describe some   important classes of nonpahological ideals.

Using 
submeasures $\phi_f$ and $\overline{\phi}_f$ which we defined at  page~\pageref{def:Erdos-Ulam-submeasure}
and using 
Proposition~\ref{prop:properties-of-asymptotic-densities}, we see that 
 the \emph{Erd\H{o}s-Ulam ideal generated by $f$} 
(\cite{just-krawczyk}),  
 $$\eu_{f}
=
\Exh(\phi_f) 
= 
\left\{A\subset\omega\ : 
\overline{\phi}_f(A)=0\right\},$$
is a nonpathological analytic P-ideal.
In particular, the the ideal          
$$\I_d=\left\{A\subseteq\omega:\limsup_{n\to\infty}\frac{|A\cap n|}{n}=0\right\}$$
of all sets of the \emph{asymptotic density} zero
is a nonpathological analytic P-ideal.

Using  measures $\mu_f$ which were  defined at page~\pageref{def:measure-for-summable-ideal}
we see
that the \emph{summable ideal generated by $f$} (\cite[Example~3, p.206]{MR0363911}),  
$$\I_f={\fin}(\mu_f)=\left\{A\subseteq\omega: \sum_{n\in A} f(n)<\infty\right\},$$ 
is a nonpathological $F_\sigma$ ideal.
In particular, the the ideal          
$$\I_{1/n}=\left\{A\subseteq\omega: \sum_{n\in A}\frac{1}{n+1}\right\}$$
is a nonpathological $F_\sigma$ ideal.

The following theorem gives a useful characterization of nonpathologicity  for  analytic  P-ideals. One of the items of this characterisation involves the \emph{Solecki ideal} $\cS$ introduced in \cite{MR1758325} (see also \cite[Section~3.6]{MR2777744}).

\begin{theorem}[{\cite[Corollary~5.6]{hrusak-katetov}}]
\label{thm:characterization-of-nonpath-of-analytic-P-ideals}    
For an analytic P-ideal $\I$ the following conditions are equivalent.
\begin{enumerate}
\item $\I$ is nonpathological.
    \item $\I\restriction A \leq_K \I_d$ for every $A\notin \I$.\label{thm:characterization-of-nonpath-of-analytic-P-ideals:item}    
\item $\cS\not\leq_K\I\restriction A$ for every $A\notin\I$.
\end{enumerate}
\end{theorem}

Another class of nonpathological ideals can be defined with the aid of infinite matrices of reals.
Namely, if a nonnegative matrix $A=(a_{i,k})_{i,k\in \omega}$ is
\emph{regular}, i.e.
$$
 \sup_{i\in \omega}\sum_{k\in\omega} a_{i,k}<\infty, \ \ 
 \lim_{i\to\infty} \sum_{k\in\omega} a_{i,k} = 1
\text{\ \  and \ }  \lim_{i\to\infty} a_{i,k}=0 \text{ for every $k\in\omega$,}$$
then 
the \emph{matrix summability ideal generated by $A$} (in short \emph{matrix ideal}), 
$$\I(A) =\Exh\left(\sup_{i\in\omega} \mu_{f_i} \right) =\left\{B\subset\N: \limsup_{i\to\infty} \sum_{k\in B} a_{i,k}=0\right\},$$
is  a nonpathological analytic P-ideal
(see \cite[the proof of Proposition~13]{MR3405547}).

Matrix ideals can be used for checking pathologicity of ideals as shown in the following theorem. 
\begin{theorem}[{\cite[Theorems~5.7, 5.14 and Proposition~5.1]{ft-fridy}}]
\label{thm:pathology-matrix}
Every nonpathological ideal can be represented as the intersection of some matrix
summability ideals i.e.~if $\I$ is a nonpathological ideal, then there exists a family $\cM$ of matrix summability ideals with  
$$\I = \bigcap \cM.$$
In particular, every nonpathological ideal is contained in some  matrix ideal.
\end{theorem}

Let us finish this section with examples  of pathological ideals.
In \cite[Theorem~4.12]{ft-mazur}, the authors constructed an $F_\sigma$ P-ideal $\I$ which is not contained in any matrix summability ideal, and consequently (by Theorem~\ref{thm:pathology-matrix}) the ideal $\I$ is pathological.
Another example of a pathological ideal can be found in \cite[Section~3.2.1]{MR4797308}. We would  emphasize that both of these examples rely heavily on the construction  of an $F_\sigma$ ideal which is not contained in any summable  ideal (given by K.~Mazur~\cite[Theorem~1.9]{mazur}). 
Using Theorem~\ref{thm:nonpath-intersection-of-summable}, one can see that the ideal constructed by Mazur  is pathological, hence his ideal seems to be the first pathological ideal constructed ever.

%%%%%%%%%%%%%%%%%%%%%%%%%%%%%%%%%%%%%%%%
%%%
%%%%%%%%%%%%%%%%%%%%%%%%%%%%%%%%%%%%%%%%

\section{Intersections of matrix ideals}
\label{sec:matrix}

The following theorem shows that item (\ref{thm:characterization-of-nonpath-of-analytic-P-ideals:item}) of Theorem~\ref{thm:characterization-of-nonpath-of-analytic-P-ideals}
characterizes ideals that can be represented as the intersections of matrix summability ideals in the realm of all ideals (not necessarily analytic P-ideals).

\begin{theorem}\label{thm:GMV-char}
\label{thm:matrix-intersection-via-KAT-and-Z}
An ideal  $\I$ is equal to the intersection of a family of matrix summability ideals if and only if $\I\upharpoonright A\leq_K \I_d$ for  every $A\not\in\I$.
\end{theorem}
\begin{proof}
`($\Rightarrow$)': Follows from \cite[Corollary~3.4]{tryba-pathology} and the proof of \cite[Theorem~3.8]{tryba-pathology} as noted by remarks under \cite[Corollary~3.9]{tryba-pathology}.

`($\Leftarrow$)': Take $A\not\in \I$. Let $f:\N\rightarrow A$ be such that for every $C\in\I\upharpoonright A$ we have $f^{-1}[C]\in \I_d$.

We will finish the proof by constructing a regular matrix $B=(b_{i,k})$ such that $A\not\in \I(B)$ and $\I\subseteq \I(B)$.  
For every $i,k\in\N$ put
$$b_{i,k}=\frac{|f^{-1}[\{k\}]\cap [1,i]|}{i} $$
Clearly, $b_{i,k}\geq 0$ for all $i,k\in\N$ and $\sum_{k=1}^\infty b_{i,k}=1$ for each $i\in\N$. Moreover, since for each $k\in\N$ we have $\{k\}\in \I\upharpoonright A$, we get $f^{-1}[\{k\}]\in\I_d$. Therefore, for each $k\in\N$ we have
$$\lim_{i\to\infty} \frac{|f^{-1}[\{k\}]\cap [1,i]|}{i}=0, $$
hence $\lim_{i\to\infty}b_{i,k}=0$. Thus, the matrix $B$ is regular.

Observe that 
$$\sum_{k\in C} b_{i,k}=\frac{|f^{-1}[C]\cap [1,i] |}{i} $$
for every $i\in \N$ and $C\subseteq\N$. 

Obviously, $f^{-1}[A]=\N$. It follows that 
$$\sum_{k\in A}b_{i,k}= \frac{|f^{-1}[A]\cap[1,i] |}{i}=\frac{|[1,i]|}{i}=1 $$
for every $i\in\N$, thus $A\not\in\I(B)$.

Moreover, since for every $C\in\I\upharpoonright A$ we have $f^{-1}[C]\in \I_d$, it follows that for every such $C$ we obtain
$$\lim_{i\to\infty}\frac{|f^{-1}[C]\cap [1,i]|}{i}=0,$$
thus $\lim_{i\to\infty} \sum_{k\in C} b_{i,k}=0$, hence $C\in \I(B)$. Therefore,  $\I\upharpoonright A\subseteq \I(B)$. Since $\I\subseteq \I\upharpoonright A\cup\cP(\N\setminus A)$ and $\N\setminus A\in\I(B)$, we obtain $\I\subseteq \I(B)$.
\end{proof}

In \cite[Question~3.13]{martinez2022pathology}, the authors asked whether the Solecki ideal $\cS$ is pathological. Below we argue that the answer is positive, which gives another  example of a pathological $F_\sigma$ ideal
(in \cite[Section~3.3]{MR4797308}, the authors announced that the same result was independently obtained by Figueroa and Hru\v{s}\'{a}k).

\begin{corollary}
        The Solecki ideal is  pathological.
\end{corollary}

\begin{proof}
It is known that the Solecki ideal $\cS$ is a tall $F_\sigma$ ideal (see e.g. \cite[Section~3.6]{MR2777744}). Suppose for the sake of contradiction that $\cS$ is nonpathological.
Then  $\cS$ would be the intersection of some family of matrix ideals
by Theorem~\ref{thm:pathology-matrix}. 
Using Theorem~\ref{thm:matrix-intersection-via-KAT-and-Z}, we obtain 
 $\cS\restriction A\leq_K\I_d$ for every $A\notin \cS$.
Thus  $\I_d$ is pathological by Theorem~\ref{thm:characterization-of-nonpath-of-analytic-P-ideals}, a contradiction.
\end{proof}

%%%%%%%%%%%%%%%%%%%%%%%%%%%%%%%%%%%%%%%%
%%%
%%%%%%%%%%%%%%%%%%%%%%%%%%%%%%%%%%%%%%%%

\section{Generalized density ideals}
\label{sec:GDI}

If $(\mu_n)_{n\in \omega}$ is  a sequence of measures on $\omega$ which are  concentrated on  finite pairwise disjoint intervals  $I_n\subseteq \omega$ (i.e.~$\{i\in \omega:\mu_n(\{i\})>0\}\subseteq I_n$ for each $n\in \omega$ and $I_n\cap I_m=\emptyset$ for $n\neq m$), then 
the \emph{density ideal generated by $(\mu_n$)}, 
$$\I_{\mu_n}=\Exh\left(\sup_{n\in \omega}\mu_n\right)=\left\{A\subseteq\omega: \limsup_{n\to\infty}\mu_n(A)=0\right\},$$
is a nonpathological analytic P-ideal (\cite{MR1711328}).
If we replace measures $\mu_n$ with submeasures $\phi_n$ in the above definition, we obtain the \emph{generalized density ideal generated by $(\phi_n)$} which are also nonpathological analytic P-ideal (\cite{MR1988247}).

In \cite{MR4124855}, the authors are interested in the following modification of density ideals.
Let $\cF$ be a family of finite subsets of $\omega$ and $f:\omega\to (0,\infty)$ be a sequence of positive reals. 
For each $F\in \cF$, we define a measure on $\omega$ 
$$\mu_{f,F}(A)= \sum_{i\in A\cap F} f(i),$$
and associated nonpathological analytic P-ideal given by 
$$\I_{f,\cF} = \exh\left(\sup_{F\in\cF} \mu_{f,F}\right).$$

The following theorem solves \cite[Problem~4.3]{MR4124855}.

\begin{theorem}
The ideal 
 of all sets of  exponential density  zero, 
$$\I_\varepsilon = \left\{A\subseteq\omega: \limsup_{n\to\infty} \frac{\ln |A\cap n|}{\ln n} =0 \right\},$$ 
is a nonpathological analytic P-ideal that is not of the form $\I_{f,\cF}$.
\end{theorem}
\begin{proof}
The ideal $\I_\varepsilon$ is a matrix summability ideal by \cite[Theorem~5.20]{ft-fridy}, so it is a nonpathological analytic P-ideal. 
Then it follows from \cite[Proposition~4.6]{tryba-gendensity} that $\I_\varepsilon$ is  a generalized density ideal, so there are pairwise disjoint intervals $I_n$ and submeasures $\psi_n:\cP(I_n)\to[0,\infty)$ such that 
$$\I_\varepsilon=\left\{A\subseteq\omega: \limsup_{n\to\infty} \psi_n(A\cap I_n)=0\right\}.$$

Assume that $\I_\varepsilon$ is of the form $\I_{f,\cF}$ for some $\cF\subseteq [\omega]^{<\omega}$ and $f:\omega\to  (0,\infty)$. Let  $\phi = \sup_{F\in\cF} \mu_{f,F}$. Since $\I_\varepsilon$ is  a generalized density ideal, we can notice that $$\I_\varepsilon=\left\{A\subseteq\omega: \lim_{n\to\infty} \varphi(A\cap I_n)=0\right\}.$$  
We also know that $\lim_{n\to\infty}\varphi(\{n\})=0$, because $\I_\varepsilon$ is tall.

Now, note that there has to be $A\not\in\I_\varepsilon$ such that $\limsup_{n\to\infty} \varphi(A\cap I_n)<\infty$. Otherwise, we would have $\I_\varepsilon=\{A\subseteq\omega: \varphi(A)<\infty\}$, which would make $\I_\varepsilon$ an $F_\sigma$ ideal, a contradiction with the fact that no tall generalized density ideal is $F_\sigma$ (see \cite[Section~3]{MR2254542}). 

Therefore,  there exists $A\not\in\I_\varepsilon$,  $\alpha>0$ and injective sequences $(n_k)\subseteq \omega$, $(F_{n_k})\subseteq \cF$ such that $\lim_{k\to\infty} \varphi(A\cap F_{n_k}\cap I_{n_k})=\alpha$. Let $B_k=A\cap F_{n_k} \cap I_{n_k}$ and $B=\bigcup_{k\in\omega} B_k$. Clearly, $B\not\in\I_\varepsilon$ as $\limsup_{n\to\infty} \varphi(B\cap I_n)=\alpha>0$. Since $\lim_{n\to\infty}\varphi(\{n\})=0$, we also get $\lim_{k\to\infty} |B_k|=\infty$. What is more, for every infinite $Z\subseteq\omega$ we have $\bigcup_{k\in Z} B_k \not\in\I_\varepsilon$. Therefore, we may assume (by trimming down $B_k$ if necessary) that $\liminf_{k\to\infty} \ln (|B_k|)/\ln {(\max B_k)}=\beta>0$.

Next, we divide every $B_k$ into $\lfloor \sqrt{|B_k|}\rfloor$ disjoint, intertwined parts $C_i^k$ (i.e. $| |C_i^k \cap n|- |C_j^k \cap n||\leq 1$ for all $i,j,n$), each with cardinality around $\sqrt{|B_k|}$. Since $B_k\subseteq F_{n_k}$, for every $X\subseteq B_k$ we have $\varphi(X)=\mu_{f,F_{n_k}}(X)$. Therefore, for every $k$, there has to be at least one $i_k$ with 
$$\varphi(C_{i_k}^k)\leq \frac{\varphi(B_k)}{\lfloor \sqrt{|B_k|}\rfloor},$$
thus for $C=\bigcup_{k\in\omega}C_{i_k}^k$ we obtain 

$$\limsup_{n\to\infty} \varphi(C\cap I_n)=\limsup_{k\to\infty} \varphi(C_{i_k}^k)\leq \limsup_{k\to\infty}\frac{\varphi(B_k)}{\lfloor \sqrt{|B_k|}\rfloor} =0,$$
because  $\limsup_{k\to\infty} \varphi(B_k)\leq \alpha$ and $\lim_{k\to\infty} \lfloor \sqrt{|B_k|}\rfloor=\infty$. 

On the other hand, since for every $k\in\omega$, 
$$|C_{i_k}^k |\geq  \left\lfloor\frac{|B_k |}{\lfloor \sqrt{|B_k| }\rfloor}\right\rfloor\geq \lfloor\sqrt{|B_k|}\rfloor$$
%and $\sqrt{|X\cup Y|}\leq \sqrt{|X|}+\sqrt{|Y|}$ for all finite $X,Y\subseteq\omega,$ 
then for every $n=\max B_k$ we have
$$\frac{\ln (|C\cap n|)}{\ln n} \geq  \frac{\ln |C_{i_k}^k|}{\ln n}\geq  \frac{\ln (\lfloor\sqrt{|B_k|}\rfloor)}{\ln n} \approx \frac{\ln (|B_k|)}{2\ln n},$$ 
hence $\limsup_{n\to\infty}\ln (|C\cap n|)/\ln n\geq \beta/2$, thus $C\not \in \I_\varepsilon$, which leads to a contradiction.
\end{proof}

%%%%%%%%%%%%%%%%%%%%%%%%%%%%%%%%%%%%%%%%
%%%
%%%%%%%%%%%%%%%%%%%%%%%%%%%%%%%%%%%%%%%%

\section{\texorpdfstring{$F_\sigma$}{F-sigma} ideals and the degree of pathology}
\label{sec:Fsigma-and-degree-of-pathol}

If $\phi$ is a lower semicontinuous  submeasure, then $\widehat{\phi}_\sigma$ is a nonpathological lower semicontinuous  submeasure  (by Proposition~\ref{prop:properties-of-phi-hat}(\ref{prop:properties-of-phi-hat:sigma-hat-is-lsc})(\ref{prop:properties-of-phi-hat:all-sets-for-lsc})).
If $P_\sigma(\phi)<\infty$ then
$$\widehat{\phi}_\sigma\leq \phi\leq P_\sigma(\phi)\cdot \widehat{\phi}_\sigma.$$
If $\phi$ is a lower semicontinuous  submeasure with $P(\phi)<\infty$, then 
$P_\sigma(\phi)=P(\phi)$ (by Proposition~\ref{prop:basic-prop-of-P-phi}(\ref{prop:basic-prop-of-P-phi:nonpathology-for-lsc})), so
$$\Exh(\phi)=\Exh\left(\widehat{\phi}_\sigma\right) \text{\ \ and\ \ } {\fin}(\phi)={\fin}\left(\widehat{\phi}_\sigma\right),$$
and consequently, both $\Exh(\phi)$ and ${\fin}(\phi)$ are  nonpathological in this case (even though $\phi$ may be pathological).

In \cite[Section~3.1]{MR4797308}, the authors 
constructed a lower semicontinuous submeasure $\phi$ such that $P(\phi)=\infty$ and  ${\fin}(\phi)$ is nonpathological. Then they asked a question \cite[Question~3.6]{MR4797308} whether  one can find $\phi$ which additionally has $P(\phi\restriction \cP(A))<\infty$ for each $A\in {\fin}(\phi)$. The following proposition answers their question in the positive.

\begin{proposition} 
\label{prop:pathologica-submeasure-with-infinite-degree-of-pathology-but-with-nonpathological-ideal}
There exists a pathological lower semicontinuous  submeasure $\phi$ on $\omega$ such that $P(\phi)=\infty$ and $\Exh(\phi)={\fin}(\phi)={\fin}$, thus ${\fin}(\phi)$ is a nonpathological $F_\sigma$  P-ideal. Moreover, $P(\phi\upharpoonright \cP(A))<\infty$ for all $A\in{\fin}(\phi)$.
\end{proposition}
\begin{proof}
By \cite[Lemma 1.8]{mazur} for every $n>0$ there exists a finite set $K_n$ and a family $\cS_n\subseteq\cP(K_n)$ such that:
\begin{enumerate}
\item $\forall A_1,\ldots, A_n \in \cS_n \ (A_1\cup \ldots \cup A_n \not = K_n)$;
\item if $p$ is a probability distribution on $K_n$ then there exists  $A\in \cS_n$ such that $p(A)\geq 1/2$.
\end{enumerate}
Assume that $\{K_n:n\in\omega\}$ is a partition of $\N$ into intervals and define $\Phi_n: \cP(K_n) \to [0,\infty)$ by
$$\Phi_n(A)=\min \left\{ |\cS|: \cS \subseteq \cS_n \textrm{ and } A\subseteq \bigcup \cS \right\} $$
for any $A\subseteq K_n$. Notice that $\Phi_n(K_n)>n$. 

Define $\phi: \cP(K_n) \to [0,\infty)$ by $\phi(A)=\sum_{n=1}^\infty \phi_n(A\cap K_n)$. Clearly, if $A$ is finite then $\Phi_n(A)=0$ for all but finitely many $n$ (thus $\phi(A)<\infty$) and if $A$ is infinite then $\Phi_n(A)\geq 1$ for infinitely many $n$ (thus $\phi(A)=\infty$). Therefore, $\Exh(\phi)={\fin}(\phi)={\fin}$.

To see that $P(\phi)=\infty$, first assume that there exists a measure $\mu:\cP(K_n) \to [0,\infty)$ such that $\mu(K_n)\geq 2$. 
Then there exists $A\in \cS_n$ such that $\mu(A)>1$, because for the probability measure $p$ given by $p(B)=\mu(B)/\mu(K_n)$ one has to find $A\in \cS_n$ such that $p(A)>1/2$, thus $\mu(A)>1$. 
On the other hand, for every $A\in\cS_n$ we have $\phi_n(A)=1$, thus $\phi(A)=1$, hence $\mu\not\leq\phi$.
It follows that for every $\mu\leq \phi$ we have $\mu(K_n)<2$. Therefore, for any measure $\mu\leq \phi$ we get
$$\frac{\phi(K_n)}{\mu(K_n)}>\frac{n}{2}, $$
for all $n$, thus $P(\phi)=\infty$.

To prove the `moreover' part, notice that for $A\in{\fin}(\phi)$, there are only finitely many sets $B\subseteq A$ for which $(\phi\upharpoonright \cP(A)) (B)>0$. Since $(\phi\upharpoonright \cP(A))(A)=\phi(A)<\infty$ (thus $(\phi\upharpoonright \cP(A))(B)<\infty$ for every $B\subseteq A$), it is clear that $P_{\fin}(\phi\upharpoonright \cP(A))<\infty$.
\end{proof}

We know  that nonpathological $F_\sigma$ ideals can be represented as intersections of matrix ideals (see Theorem~\ref{thm:pathology-matrix}). We can possibly strengthen it a bit.
\begin{theorem}
\label{thm:nonpath-intersection-of-summable}
If $\I$ is a nonpathological $F_\sigma$ ideal then it can be represented as the intersection of summable ideals.
In particular, every nonpathological $F_\sigma$ ideal is contained in some summbale ideal.
\end{theorem}
\begin{proof}
Let $\phi$ be a nonpathological submeasure such that $\I={\fin}(\phi)$. Take $X\not\in\I$. Then $\phi(X)=\infty$, thus, since $\phi$ is lower semicontinuous and nonpathological, there exist finite sets $A_1,A_2,\ldots\subseteq X$ and measures $\mu_1,\mu_2,\ldots$ with $\mu_n\leq \phi$ for every $n$ such that $\mu_n(A_n)\geq 2^n$. 

We will finish the proof by finding a summable ideal $\I_g$ such that $\I\subseteq\I_g$ and $X\not\in\I_g$.
Define $g:\N\rightarrow[0,\infty)$ by
$$g(i)=\sum_{n=1}^{\infty}\frac{\mu_n(\{i\})}{2^n}. $$
Then for every $A\subseteq\N$ we have
$$\sum_{i\in A} g(i) =\sum_{i\in A} \sum_{n=1}^{\infty}\frac{\mu_n(\{i\})}{2^n}=\sum_{n=1}^{\infty} \frac{\mu_n(A)}{2^n}\leq\sum_{n=1}^{\infty} \frac{\phi(A)}{2^n}=\phi(A). $$
It follows that if $\phi(A)<\infty$ then $\sum_{i\in A} g(i)<\infty$, thus $\I\subseteq \I_g$.

Now, observe that 
$$\sum_{i\in X} g(i)=\sum_{n=1}^{\infty} \frac{\mu_n(X)}{2^n}\geq \sum_{n=1}^{\infty} \frac{\mu_n(A_n)}{2^n}\geq \sum_{n=1}^{\infty} 1=\infty, $$
thus $X\not\in\I_g$.
\end{proof}

\begin{corollary}
If $\I$ is a nonpathological $F_\sigma$ ideal then for every $X\not\in\I$, $\I\upharpoonright X$ can be extended to some summable ideal.
\end{corollary}

\begin{proposition}
Let $\I$ be an ideal. Then the following are equivalent.
\begin{enumerate}
\item $\I$ is an intersection of summable ideals.
\item for every $A\not\in\I$ there exists some summable ideal $\I_f$ such that $\I\restriction A\leq_K \I_f$.
\end{enumerate}
\end{proposition}
\begin{proof}
The $(1)\Rightarrow (2)$ implication is obvious.

To prove the $(2)\Rightarrow (1)$ implication, take $A\not\in\I$ and a summable ideal  $\I_f$ such that $\I\restriction A\leq_K \I_f$. Let $g:\N\to A$ be a  witness to that. 
Define the function $h:\N\to[0,\infty)$ by
$$h(n)=\sum_{i\in g^{-1}(\{n\} )}f(i) $$
for every $n\in\N$. Clearly, $\sum_{n\in\N} h(n)=\infty$, thus $\I_h$ is a summable ideal such that $A\not\in\I_h$.

We will finish the proof by showing that $\I\restriction A\subseteq\I_h$. Take $B\in\I\upharpoonright A$. Then $g^{-1}(B)\in\I_f$, therefore $\sum_{i\in g^{-1}(B) }f(i)<\infty$, thus $\sum_{n\in B} h(n)<\infty$. It follows that $B\in\I_h$.
\end{proof}

\begin{proposition}
If $\I$ is an $F_\sigma$ ideal that can be represented as an intersection of summable ideals then for every $X\not\in\I$ there exists a submeasure $\phi$ such that ${\fin}(\phi)=\I$ and $\widehat\phi_\sigma(X)=\infty$. 
\end{proposition}
\begin{proof}
Let $X\not\in\I$ and let $\psi$ be such that ${\fin}(\psi)=\I$. Since there exists a summable ideal $\I_f$ such that $\I\subseteq\I_f$ and $X\not\in\I_f$, for a measure $\mu_f(A)=\sum_{n\in A}f(n)$ we have $\mu_f(X)=\infty$. Define $\phi=\max\{\psi,\mu_f\}$. Then $\phi$ is a lower semicontinuous submeasure as a maximum of a measure and a lower semicontinuous submeasure. Observe that $\widehat\phi_\sigma\geq \mu_f$, because $\mu_f$ is a $\sigma$-measure and $\mu_f\leq \phi$. Hence $\widehat\phi_\sigma(X)=\infty$. 

To see that ${\fin}(\phi)=\I$, first note that $\psi\leq\phi$, thus ${\fin}(\phi)\subseteq{\fin}(\psi)=\I$. On the other hand, if $A\in\I\subseteq\I_f$, then we have both $\psi(A)<\infty$ and $\mu_f(A)<\infty$. Therefore, $\phi(A)=\max\{\psi(A),\mu_f(A)\}<\infty$.
\end{proof}

Note that summable ideals are nonpathological $F_\sigma$ ideals, so they can be represented as intersections of matrix ideals. Moreover,  $\I_d$ is a matrix ideal that cannot be extended to any summable ideal, thus being represented as an intersection of summable ideals is strictly stronger than being  represented as  an intersection of matrix ideals. However, the following question is still open.

\begin{question}
Let $\I$ be an $F_\sigma$ ideal. Then each of the following items implies the next.
\begin{enumerate}
\item $\I$ is nonpathological.
\item $\I$ can be represented as an intersection of summable ideals.
\item $\I$ can be represented as  an intersection of matrix ideals.
\end{enumerate}
Can these implications be reversed?
\end{question}

We finish this section with an example which gives us answer to  \cite[Question~3.10]{MR4797308}.
First, we need to introduce some notions and notations necessary to the formulation of this question.
For any finite set $K\subseteq \omega$ and a family  $\cF\subseteq \cP(K)$ with $\bigcup \cF = K$, we define 
 the \emph{covering number of $\cF$ in $K$} by
$$\delta(K,\cF) = \frac{\min\{ |\{F\in \cF:i\in F\}|:i\in K\}}{|\cF|}.$$

Let $\phi$ be a lsc submeasure such that $\phi(\omega)=\infty$ and there is $M>0$ such that $$\cB_M=\{A:\phi(A)<M\}$$
is a cover of $\omega$.
Let $(K_n:n< \omega)$ be a strictly increasing sequence of finite subsets of $\omega$ such that $\bigcup_{n\in \omega}K_n=\omega$ and let  $(\cF_n:n<\omega)$ be a sequence such that  $\cF_n\subseteq\cP(K_n)\cap \cB_M$ and $\bigcup \cF_n=K_n$ for each $n$.
Then we define
$$\delta(\phi,M,(K_n),(\cF_n)) = \inf\{\delta(K_n,\cF_n):n\in \omega\}.$$

\begin{lemma}
\label{lem:asljfdasjdfasd}
    If ${\fin}(\phi)$ is a tall ideal, then   there exists $M>0$ with
    $$\delta(\phi,M,(K_n),(\cF_n)=0,$$
where
$K_n=\{i\in \omega:i<n\}$ and $\cF_n=\{\{i\}:i<n\}$ for each $n\in \omega$.
\end{lemma}

\begin{proof}
    Since ${\fin}(\phi)$ is a tall ideal, there is $M>0$ such that $\phi(\{n\})\leq M$ for each $n\in \omega$.
    Then $(K_n)$ is a strictly increasing sequence of finite subsets of $\omega$ with $\bigcup_{n\in \omega}K_n=\omega$,   $\cF_n\subseteq \cP(K_n)\cap \cB_M$ and $\bigcup \cF_n=K_n$ for each $n<\omega$.
Moreover,
$\delta(K_n,\cF_n) = 1/n$ for each $n$.
Thus, $\delta(\phi,M,(K_n),(\cF_n))=0$.    
\end{proof}

In 
\cite[Theorem~3.9]{MR4797308}, the authors proved that 
if  
$\delta(\phi,M,(K_n),(\cF_n))>0$, then $\widehat{\phi}_\sigma(\omega)<\infty$. Then they  
asked
\cite[Question~3.10]{MR4797308}
whether  $\widehat{\phi}_\sigma(\omega)<\infty$ implies $\delta(\phi,M,(K_n),(\cF_n))>0$, and whether 
$\delta(\phi,M,(K_n),(\cF_n))=0$ implies that ${\fin}(\phi)$ is nonpathological.
Below we show that the answers to both questions are negative.

To answer the first question, we take  a submeasure $\psi_3$ which is defined in 
  Proposition~\ref{prop:pathological-lsc-submeasures}(\ref{prop:pathological-lsc-submeasures:item-non-lsc-hat}). We have already proved there  that  
$\psi_3$ is pathological and $\widehat{(\psi_3)}_\sigma(\omega)<\infty$. 
Once we prove that $\Fin(\psi_3)$ is tall,  Lemma~\ref{lem:asljfdasjdfasd} will show that
$\delta(\phi,M,(K_n),(\cF_n))=0$
for an appropriate $M$, $(K_n)$ and $(\cF_n)$, so the answer to the first question will be obtained. 

Tallness of $\Fin(\psi_3)$  follows from the fact that for every $n\in \omega$ and for every $i\in I_n$ we have $\psi_3(\{i\})\leq (n+1)/2^{n+1}$.  
To prove this fact, suppose for the sake of contradiction that 
there is  $n\in \omega$ and there is  $i\in I_n$ such that  $\psi_3(\{i\})> (n+1)/2^{n+1}$.  
Then we take the measure $\mu=\phi_n(\{i\})\delta_{i}$ and observe that $\mu\leq \phi_n$, so $\mu(\I_n)\leq 1/2^{n+1}$. 
On the other hand, $\mu(I_n)\geq \mu(\{i\})=\phi_n(\{i\})=\psi_3(\{i\})/(n+1)>1/2^{n+1}$, a contradiction.

To answer the second question, we take a pathological $F_\sigma$ P-ideal $\I$ constructed in \cite[Theorem~4.12]{ft-mazur}. Since $\I$ is tall,
Lemma~\ref{lem:asljfdasjdfasd}  show that
$\delta(\phi,M,(K_n),(\cF_n))=0$
for an appropriate $M$, $(K_n)$ and $(\cF_n)$, so the answer to the second question is obtained. 

%%%%%%%%%%%%%%%%%%%%%%%%%%%%%%%%%%%%%%%%%%%%%%%%%%%
%%%%%%%%%%%%%%%%%%%%%%%%%%%%%%%%%%%%%%%%%%%%%%%%%%%
%%%%%%%%%%%%%%%%%%%%%%%%%%%%%%%%%%%%%%%%%%%%%%%%%%%

\section{Van der Waerden ideal}
\label{sec:vdW}

We start this section by recalling two versions of van der Waerden Theorem (see e.g.~\cite[Section~2]{MR1044995}).
\begin{theorem}\label{thm:vdw-infinite}
Let $A\subseteq\N$ be a set containing arithmetic progressions of any given length. For any partition of $A$ into finitely many subsets, at least one of them contains  arithmetic progressions of any given finite length.
\end{theorem}
\begin{theorem}\label{thm:vdw-finite}
For any $n\in\N$ there exists a number $W_n\in\N$ such that for any set $A\subseteq\N$ containing an arithmetic progression of length $W_n$ if $A=B\cup C$, then at least one of the sets $B$, $C$ contains an arithmetic progression of length $n$.
\end{theorem}
Due to the first of these theorems, the family $\cW$ which consists of those subsets of $\N$ that do not contain arithmetic progressions  of arbitrary finite length
% $$\cW= \{A\subseteq\N: \text{$\exists {n\in\N}$  $A$ does not contain an arithemtic progression of length $n$}\} $$
is an ideal called \emph{van der Waerden ideal}.
It is known that $\cW$ is an $F_\sigma$ ideal (\cite[Example~4.12]{MR4572258}) and that it can be extended to a summable ideal (\cite[Theorem~3.1]{klinga-nowik}), thus by the homogeneity
of $\cW$ (\cite[Example~2.6]{kwela-tryba}) it is equal to the intersection of a family of  summable ideals.
To the best of our knowledge, the possible pathology of $\cW$ (or lack of thereof)  has not  been researched before.

To prove the nonpathology of $\cW$, we will also need a finitary version of Szemer\'edi's Theorem (\cite{szemeredi}, see also \cite[Section~1.4]{MR1044995}).

\begin{theorem}[{\cite{szemeredi}}]\label{thm:szemeredi}
Let $n\in\N$ and let $0<\delta\leq 1$. Then there exists a number $N(n,\delta)$ such that for any  arithmetic progression $A$ of length $k\geq N(n,\delta)$ every  set $B\subseteq A$ of size at least $\delta k$ contains an arithmetic progression of length $n$.
\end{theorem}

\begin{theorem}
$\cW$ is a nonpathological $F_\sigma$ ideal.
\end{theorem}
\begin{proof}
We will obtain the theorem by finding an lsc submeasure $\phi$ such that ${\fin}(\phi)=\cW$ and $P_{\fin}(\phi)\leq 2$.

To begin, we will define inductively the sequence $(V_n)$. Let $V_1=1$, $V_2=2$ and assume that we have defined $V_1,\ldots,V_{n-1}$ for some $n\geq3$. 

Then for every $i< n$ we define as $N_i$ the smallest natural number such that for every   arithmetic progression $A$ of length at least $N_i$ and every set $B\subseteq A$ such that 
$$\frac{|B|}{|A|}\geq \frac{i}{n} $$
the set $B$ has to contain an arithmetic progression of length $V_i$. Such $N_i$ exists by Theorem~\ref{thm:szemeredi} (we put $N_i=N(V_i,i/n)$). 

We pick as $V_n$ the smallest natural number such that $V_n\geq N_i$ for every $i<n$ and $V_n\geq W_{V_{n-1}}$, where $W_j$ is defined as in Theorem~\ref{thm:vdw-finite}.

We can now proceed to defining the function $\phi:\cP(\N)\rightarrow [0,\infty]$ by
$$\phi(A)=\sup\{n\in\N: A \text{ contains an arithmetic progression of length } V_n \}  $$
for every $A\subseteq\N$.

Clearly, $\phi(\emptyset)=0$, $\phi(\N)=\infty$ and $\phi(A)\leq\phi(B)$ whenever $A\subseteq B$.

We will prove that $\phi(A\cup B)\leq \phi(A)+\phi(B)$. We may assume that $A,B$ are nonempty and disjoint. We have two cases. 

If  $\phi(A\cup B)=\infty$ then $A\cup B$ contains an arithmetic progression of arbitrary length, thus  by Theorem~\ref{thm:vdw-infinite} at least one of the two sets $A,B$ also needs to contain an arithmetic progression of any given length, thus at least one of $\phi(A)$, $\phi(B)$ has to be infinite, hence  $\phi(A) + \phi(B)=\infty$.

If  $\phi(A\cup B)=n$ for some $n\in\N$ then $A\cup B$ contains an arithmetic progression of length $V_n$. Since $V_n\geq W_{V_{n-1}}$, at least one of the two sets $A,B$ contains an arithmetic progression of length $V_{n-1}$ -- we may assume that it is $A$. Since $B$ is nonempty, $\phi(B)\geq 1$. Hence, $\phi(A)+\phi(B)\geq (n-1)+1=n$.

Therefore, $\phi$ is a submeasure. It is easy to see that it is lsc and that ${\fin}(\phi)=\cW$.  

To finish the proof, we need to show that $P_{\fin}(\phi)\leq 2$.
Take a finite, nonempty set $A\subseteq\N$. Then $\phi(A)=n$, hence it contains an arithmetic progression $B$ of length $V_n$.

If $n\leq 2$, then the measure $\mu(C)=|C\cap B|$ fulfills $\mu\leq\phi$ and $\mu(B)=\mu(A)=\phi(A)$, so we are done.

If $n>2$, we define the measure $\nu:\cP(A)\rightarrow[0,\infty)$  by 
$$\nu(C)=\frac{n\cdot |C\cap B|}{V_n}$$
for every $C\subseteq A$. 
Observe that $\nu(A)=\nu(B)=n$. 
We may also notice that for any $i\leq n$ we have that $\phi(C)\geq i$ whenever $\nu(C)\geq i$. Indeed, if  $\nu(C)\geq i$ then 
$$\frac{|C\cap B|}{|B|}\geq \frac{i}{n}. $$
Since $B$ is an arithmetic progression of length $V_n$, by the fact that $V_n\geq N(V_i,i/n)$ we find that $C$ contains an arithmetic progression of length $V_i$.

It follows that for every $C\subseteq A$ we have
$$\frac{\nu(C)}{\phi(C)}\leq \frac{i+1}{i}\leq 2.  $$

Therefore, if we take the measure  $\mu:\cP(A)\rightarrow[0,\infty)$  given by $\mu(C)=\nu(C)/2$ we obtain $\mu\leq \phi$. Since $\mu(A)= n/2$, we get 
$$\frac{\phi(A)}{\mu(A)}=2,$$
thus $P_{\fin}(\phi)\leq 2$.
\end{proof}

Let us recall one more   $F_\sigma$ ideal which is also defined with the aid of Ramsey theory. Namely, using Folkman's theorem (see e.g.~\cite[Section~3.4]{MR1044995}) we obtain the \emph{Folkman ideal} (\cite{kwela-tryba})
$$\cF = \{A\subseteq \N: \exists n \, \forall D\subseteq \N\ (|D|=n\implies FS(D)\not\subseteq A)\},$$
where $FS(D)$ is the set of all sums of distinct elements of $D$.
We do not know whether $\cF$ is a nonpathological ideal.

%%%%%%%%%%%%%%%%%%%%%%%%%%%%%%%%%%%%%%%%%%%%%%%%%%%%
%%%%%%%%%%%%%%%%%%%%%%%%%%%%%%%%%%%%%%%%%%%%%%%%%%%%
%%%%%%%%%%%%%%%%%%%%%%%%%%%%%%%%%%%%%%%%%%%%%%%%%%%%

\section{Josefson-Nissenzweig property}
\label{sec:JNP}

Let $X$ be a Tychonoff space. 
By $C(X)$
we denote the set of all continuous real-valued functions on $X$,
and by $C^*(X)$ we denote the subspace of $C(X)$ consisting of all bounded functions.
% By $C_p(X)$
% we denote the space of all continuous real-valued functions on $X$ endowed with the
% pointwise topology. 
% By $C^*_p(X)$ we denote the subspace of $C_p(X)$ consisting of all bounded functions.

By a \emph{measure on $X$} we mean a $\sigma$-additive measure defined on the $\sigma$-algebra of all Borel subsets of $X$ which is regular, signed  and has bounded total variation.
We say that a measure $\mu$ on $X$ 
is \emph{finitely supported} if  $\mu=\sum_{i=1}^k a_{i} \delta_{x_{i}}$ for some $x_{i}\in X$ and $a_i\in\R$, where $\delta_{x}$ is the probability measure concentrated on $x$. 
In the case $\mu$ is finitely supported, its variation is given by  
 $$||\mu||=\sum_{i=1}^k |a_i|$$ 
 and we will write 
 $$\mu(f)=\int_X f\,d\mu=\sum_{i=1}^k a_i f(x_i)$$ 
 for every $f\in C(X)$.

\begin{definition}[{\cite{marciszewski-sobota}}]
A Tychonoff space $X$ has the \emph{Josefson-Nissenzweig property} (the
\emph{bounded Josefson-Nissenzweig property}, resp.), or shortly the JNP (the BJNP, resp.), if
$X$ admits a sequence $(\mu_n)$ of 
finitely supported measures on $X$  such that $||\mu_n||=1$ for every $n\in\N$  and $\lim_{n\to\infty}\mu_n( f ) = 0$ for every $f \in C(X)$ (resp. $f \in C^*(X)$).
\end{definition}

Clearly, if $Y$ is a subspace of $X$ and $Y$ has the JNP (BJNP, resp.) then $X$ has the JNP (BJNP, resp.) too.

For an ideal $\I$ on $\N$, consider the space $X(\I)=\N\cup\{\infty\}$ with the topology such that every point in $\N$ is isolated while every open neighborhood  of $\infty$ is of the form $\{\infty\}\cup (\N\setminus A)$ for some $A\in\I$.

The JNP and BJNP  for the space $X(\I)$ were recently investigated in \cite{marciszewski-sobota}.

\begin{theorem}[{\cite[Theorem~A]{marciszewski-sobota}}]
Let $\I$ be an ideal    on $\N$. The following conditions are equivalent.
\begin{enumerate}
\item $X(\I)$ has the JNP.
\item $X(\I)$ contains a non-trivial convergent sequence.
\item $\I\approx_K {\fin}$.
\end{enumerate}
\end{theorem}

\begin{theorem}[{\cite[Theorem~C]{marciszewski-sobota}}]
\label{thm:ms-bjnp}
Let $\I$ be an ideal on $\N$. The following conditions are equivalent.
\begin{enumerate}
\item $X(\I)$ has the BJNP.
\item There is a matrix ideal $\J$ such that $\I\subseteq\J$.
\end{enumerate}
\end{theorem}
Note that originally in  \cite{marciszewski-sobota} the space $X(\I)$ had the BJNP if and only if $\I$ could be extended to a density ideal, but extension to a density ideal was proved to be equivalent to extension to a matrix ideal in \cite[Section~3]{tryba-pathology}.

We say that a family $\cA$ of subsets of $\N$ is $\I$-almost disjoint if $\cA\subseteq\I^+$ and $A\cap B\in\I$ for all distinct $A,B\in\cA$.

Let $\I$ be an ideal and $\cA$ be an nonempty $\I$-almost disjoint family of subsets of $\N$. Consider the  space $\Psi_\I (\cA)=\N\cup\cA$ 
with the topology such that every point in $\N$ is isolated while every open neighborhood  of $A\in\cA$ is of the form $\{A\}\cup A\setminus B$, where $B\in\I$.
If $\I=\Fin$, the space $\Psi_{\Fin}(\cA)=\Psi(\cA)$ is known as \emph{Mr\'{o}wka-Isbell space} (see e.g. \cite{MR3822423}).

We can  apply Theorem~\ref{thm:ms-bjnp} to spaces $\Psi_\I(\cA)$.

\begin{corollary}
An ideal  $\I$   on $\N$ is an intersection of matrix ideals $\Longleftrightarrow$
$\Psi_\I(\cA)$ has the BJNP  for every $\I$-almost disjoint family $\cA$.
\end{corollary}
\begin{proof}
$(\impliedby)$: Let $A\not \in \I$ and $\cA=\{A\}\subseteq\I^+$. Then $\Psi_\I(\cA)$ is equal to $X(\J_A)$ with $\J_A = \{B\subseteq \omega: A\cap B\in \I\}$. Thus    $X(\J_A)$ has   the BJNP, hence by Theorem~\ref{thm:ms-bjnp}  the ideal $\J_A$ can be extended to a matrix ideal. Therefore, $\I$ is an intersection of matrix ideals.

$(\implies)$: Let $\cA\subseteq \I^+$ be an $\I$-almost disjoint family and take $A\in\cA$. Put $\cB=\{A\}\subseteq\I^+$. Since $\I$ is an intersection of matrix ideals, $\I\restriction A$ can be extended to a matrix ideal, so $X(\I\restriction A)$, has the BJNP, which means that  $\Psi_\I(\cB)$ has the BJNP as well. Moreover, $\Psi_\I(\cB)$ is a subspace of $\Psi_\I(\cA)$, hence $\Psi_\I(\cA)$ has the BJNP.
\end{proof}

\begin{theorem}\label{thm:oplus-bjnp}
Let $(Y_j)_{j\in J}$ be pairwise disjoint Tychonoff spaces. The following conditions are equivalent.
\begin{enumerate}
    \item $X=\bigsqcup\{Y_j : j\in J\}$ has the JNP (BJNP, resp.).
    \item  $Y_j$ has the JNP (BJNP, resp.) for some $j\in J$.
\end{enumerate}
\end{theorem}
\begin{proof}
$(2)\Longrightarrow (1)$: Obvious, because every $Y_j$ is a subspace of $X$.

$(1)\Longrightarrow (2)$:   Let $\{\{x_i^n:i\leq k_n\}:n\in\N\}$ and  $\{\{a_i^n:i\leq k_n\}:n\in\N\}$ be the witness for the JNP (BJNP, resp.), i.e. for every $f\in C(X)$ ($f\in C^*(X)$, resp.) we have 
$$\lim_{n\to\infty}\sum_{i\leq k_n} a_i^n f(x_i^n)=0 \quad  \text{ and }\quad \sum_{i\leq k_n}|a_i^n|=1 \text{ for every $n\in\N$}. $$  
For each $n\in\N$, denote the set $\{x_i^n:i\leq k_n\}$ by $X_n$, the set $\{x_i^n: a_i^n>0\}$ by $X_n^+$ and the  set $\{x_i^n: a_i^n<0\}$ by $X_n^-$.

For each $n\in\N$ define the measure $\delta_n:\cP(X)\rightarrow [0,1]$ by 
$$\delta_n(A)=\sum_{\{i\leq k_n: x_i^n\in A\}} |a_i^n|$$
for each $A\subseteq X$. We have two cases: either there exists $j_0\in J$ such that $\delta_n(Y_{j_0})$ does not tend to zero or $\lim_{n\to\infty}\delta_n(Y_j)=0$ for all $j\in J$.

In the first case, we will show that $Y_{j_0}$ has the JNP (BJNP, resp.). Assume otherwise and pick $\alpha>0$ and a sequence $n_1<n_2<\ldots$ such that 
$$\lim_{l\to\infty}\delta_{n_l}(Y_{j_0})=\alpha \quad \text{ and } \quad \delta_{n_l}(Y_{j_0})>0 \text{ for each $l\in\N$}.$$

Since $Y_{j_0}$ does not have the JNP (BJNP, resp.), there exists a function $g\in C(Y_{j_0})$ ($g\in C^*(Y_{j_0})$) such that 
$$\limsup_{l\to\infty} \left| \sum_{\left\{i\leq k_{n_l}: x_i^{n_l}\in Y_{j_0} \right\}} \frac{a_i^{n_l} g(x_i^{n_l})}{\delta_{n_l} (Y_{j_0})}    \right| >0.  $$
Define the function $f:X\to \R$ by
$$f(x)=
    \begin{cases}
    g(x)&\text{if } x\in Y_{j_0}, \\    
    0&\text{otherwise.}
    \end{cases}$$
Clearly, $f\in C(X)$ ($f\in C^*(X)$). However,
$$  \left| \sum_{i\leq k_{n_l}} a_i^{n_l} f(x_i^{n_l})    \right|=   \left| \delta_{n_l} (Y_{j_0}) \sum_{\left\{i\leq k_{n_l}: x_i^{n_l}\in Y_{j_0} \right\}} \frac{a_i^{n_l} g(x_i^{n_l})}{\delta_{n_l} (Y_{j_0})}    \right|,$$
which does not tend to zero, a contradiction.

In the second case, observe that for every $n\in\N$ we have $\delta_n(X_n^+)\geq 1/2$ or  $\delta_n(X_n^-)\geq 1/2$. Put
$$Z_n=
    \begin{cases}
    X_n^+&\text{if } \delta_n(X_n^+)\geq 1/2, \\    
    X_n^-&\text{otherwise.}
    \end{cases}$$
We will obtain the thesis by showing that the second case contradicts the assumption that $X$ has the JNP (BJNP, resp.). We will do so by finding a function $f\in C^*(X)$ such that for infinitely many $n\in\N$ we have
$$\left|\sum_{i\leq k_n} a_i^n f(x_i^n)\right|\geq \frac{1}{4} $$

We will now inductively construct sequences of natural numbers $(n_m)$, sets $(D_m)$ with $\lim_{n\to\infty}\delta_n(D_m)=0$ and functions $(f_m)$ with $\dom(f_m)\subseteq X$ 
and $\ran(f_m)\subseteq[0,1]$.

First, let $n_1=1$ and $D_1=\bigcup \{Y_j: Y_j\cap X_1\not =\emptyset\}$. Because each $Y_j$ is Tychonoff and $X_1$ is finite, there exists a continuous function $f_1:D_1\rightarrow [0,1]$ such that $f_1(x)=1$ for $x\in Z_1$ and $f_1(x)=0$ for $x\in X_1\setminus Z_1$.
%We can notice that
%$$\left|\sum_{i\leq k_1} a_i^1 f_1(x_i^1)\right|=\left| \sum_{\{i\leq k_1:x_i^1\in Z_1 \}}a_i^1 \right|=\delta_1(Z_1)\geq\frac{1}{2}.$$
Since $X_1$ is finite, $D_1$ is made of finitely many $Y_j$, thus $\lim_{n\to\infty} \delta_n(D_1)=0$ as $\delta_n(Y_j)$ tends to zero for every $j\in J$.

Suppose we have constructed $n_m, f_m, Z_m$ for $m< l$ for some $l\in\N$. Then we put as $n_{l}$ such $n\in\N$ that $\delta_n(D_{l-1})< 1/8$. Let $D_l=\bigcup \{Y_j: \bigcup_{i\leq n_l} X_i\cap Y_j\not =\emptyset\}$.
Because $f_{l-1}$ is continuous on $D_{l-1}$, which is a closed subset of $D_l$, $X_{n_l}$ is finite and each $Y_j$ is Tychonoff, there exists a continuous function $f_l:D_l\rightarrow [0,1]$ such that 
\begin{itemize}
\item $f_l\upharpoonright D_{l-1}=f_{l-1}$;
\item $f_l(x)=1$ for $x\in Z_{n_l}\setminus D_{l-1}$;
\item $f_l(x)=0$ for $x\in (X_{n_l}\setminus Z_{n_l}) \setminus D_{l-1}$.
\end{itemize}
Once again, $D_l$ is made of finitely many $Y_j$, thus $\lim_{n\to\infty} \delta_n(D_l)=0$.

We can now define the function $f:X\rightarrow[0,1]$ in such way that 
\begin{itemize}
\item $f\upharpoonright D_{l}=f_{l}$ for every $l\in\N$;
\item $f(x)=0$ for $x\in X \setminus\bigcup_{l\in\N} D_{l}$.
\end{itemize}
It is easy to see that $f\in C^*(X)$ as each $f_l$ is continuous while $\bigcup_{l\in\N} D_{l}$ and each $D_l$  are all clopen subsets of $X$.

On the other hand, for all $l\geq 2$ we obtain
\begin{equation*}
\begin{split}
\left|\sum_{i\leq k_{n_l}} a_i^{n_l} f\left(x_i^{n_l}\right)\right|
&\geq 
\left| \sum_{\left\{i\leq k_{n_l}:x_i^{n_l}\in D_{l}\setminus D_{l-1} \right\}}a_i^{n_l}f\left(x_i^{n_l}\right) \right| -\left| \sum_{\left\{i\leq k_{n_l}:x_i^{n_l}\in  D_{l-1} \right\}}a_i^{n_l} f\left(x_i^{n_l}\right) \right|
\\&\geq
\left| \sum_{\left\{i\leq k_{n_l}:x_i^{n_l}\in Z_{n_l}\setminus D_{l-1} \right\}}a_i^{n_l} \right| - \sum_{\left\{i\leq k_{n_l}:x_i^{n_l}\in  D_{l-1} \right\}} \left|a_i^{n_l}  \right|
\\&=
\delta_{n_l}(Z_{n_l}\setminus D_{l-1}) - \delta_{n_l}( D_{l-1})\geq\left(\frac{1}{2}-\frac{1}{8}\right)-\frac{1}{8}=\frac{1}{4}.
\end{split}
\end{equation*}
\end{proof}

\begin{corollary}
Let $\cA$ be a countable $\I$-almost disjoint family. Then the following conditions are equivalent.
\begin{enumerate}
\item $\Psi_\I(\cA)$ has the BJNP.

\item There exists $A\in\cA$ such that  $X(\I\upharpoonright A)$ has the BJNP.

\item There exists $A\in\cA$ such that $\I\upharpoonright A$ extends to a matrix ideal.
\end{enumerate}
\end{corollary}
\begin{proof}

$(2)\Longleftrightarrow(3)$: follows from Theorem~\ref{thm:ms-bjnp}.

$(1)\Longleftrightarrow(2)$: follows from Theorem~\ref{thm:oplus-bjnp} as for every countable $\I$-almost disjoint family $\cA=\{A_n:n\in\N\}$ there exists a countable pairwise disjoint family $\cB=\{B_n:n\in\N\}$ such that $\Psi_\I(\cA)$ is homeomorphic to $\bigsqcup_{n\in \N} X(\I\upharpoonright B_n)$ while $X(\I\upharpoonright B_n)$ is homeomorphic to $X(\I\upharpoonright A_n)$ for every $n\in\N$. 
\end{proof}

%%%%%%%%%%%%%%%%%%%%%%%%%%%%%%%%%%%%%%%%
%%%
%%%%%%%%%%%%%%%%%%%%%%%%%%%%%%%%%%%%%%%%

\section{The intersection of matrix ideals can be non-Borel}
\label{sec:non-borel-intersection}

In this section, we will prove  the consistency of existence of a non-Borel ideal that is an intersection of matrix ideals (Corollary~\ref{cor:non-Borel-intersection-of-matrix}), which consistently answers \cite[Question~3]{ft-mazur}.
The consistency is expressed with the aid of 
the \emph{tower number}  $\mathfrak{t}$ which is the smallest length of a tower of infinite subsets of $\omega$, and in the proof we will also use the \emph{almost disjointness number} $\mathfrak{a}$ which is the smallest cardinality of any MAD family on $\omega$  (i.e.~an infinite maximal almost disjoint family of infinite subsets of $\omega$).
For more on these cardinals, see e.g.~\cite{MR2768685}.

For a MAD family  $\cA$, we write  $\I(\cA)$ to  denote the ideal generated by $\cA$:
$$B\in \I(\cA)\iff B\setminus (A_1\cup \dots\cup A_n) \text{ is finite for some } A_1,\dots,A_n\in \cA.$$

In \cite[Theorem~3.9]{MR2017358}, the authors proved that 
assuming $\mathfrak{t}=\continuum$, there  exists a MAD family $\cA$ such that $\I(A)$ is K-homogeneous. 
We will use a slightly modified version of the above result -- we will additionally require  that the constructed almost disjoint family consists of sets from a tall ideal (see Lemma~\ref{lem:hr1} and Theorem~\ref{thm:MAD-K-homo}). The proof will be almost the same as in the original result, but we will provide it for the sake of completeness.

Throughout this section, we will assume that every family $\cF\subseteq\omega^\omega$ consists of injections.
For functions $f,g:X\rightarrow Y$ we write $f=^*g$ if $f(x)=g(x)$ for all but finitely many $x\in X$. Moreover, we write $A\subseteq^* B$ if $B\setminus A$ is finite.   

Let $\cA$ be an almost disjoint family and let $\cF\subseteq \omega^\omega$.
We say that  $\cA$ \emph{respects} $\cF$ if $f^{-1}[A]\in\I(\cA)$ for all $f\in\cF$ and all $A\in\cA$.

\begin{lemma}\label{lem:hr1}
Let $\I$ be a tall ideal. Let $\cA\subseteq\I$ be an almost disjoint family that respects $\cF\subseteq\omega^\omega$. 
Assume also that $|\cA|<\mathfrak{t}$,  $|\cF|<\mathfrak{t}$ and $X\in\cP(\omega)\setminus \I(\cA)$.
Then there exists an almost disjoint family $\cB\subseteq\I$ such that $\cA\subseteq\cB$, $|\cB|<\mathfrak{t}$, $\cB$ respects $\cF$ and $\cB\cap [X]^\omega\not=\emptyset$.
\end{lemma}
\begin{proof}
If $\cA\cap [X]^\omega\not=\emptyset$ then we can put $\cB=\cA$ and we are done. Hence, we can assume that $\cA\cap [X]^\omega=\emptyset$.

First, notice that since $X\not\in\I(\cA)$ and $|\cA|<\mathfrak{t}\leq\mathfrak{a}$, the set $\{A\cap X: A\in\cA \land |A\cap X|=\omega\}$ cannot be a MAD family on $X$,
so there exists a set $Y\in[X]^\omega$ such that $Y\cap A\in{\fin}$ for every $A\in\cA$. 
In particular, $Y\notin \I(\cA)$. 
It follows that we can assume without loss of generality that  $A\cap X\in{\fin}$ for every $A\in\cA$.

Second, we can assume without loss of generality that the identity function $f(n)=n$ belongs to $\cF$.

Let $\cF'$ be the closure of $\cF$ under compositions. Observe that $\cA$ respects $\cF'$. Indeed, for any $f,g\in \cF$ and $A\in\cA$ we have
$$(f\circ g)^{-1}[A]=g^{-1}\left[f^{-1}[A]\right]\subseteq \bigcup_{i=1}^n g^{-1} [A_i]  $$
for some $A_i\in \cA$ as $f^{-1}[A]\in\I(\cA)$. Since $g^{-1} [A_i]\in\I(\cA)$, we obtain $(f\circ g)^{-1}[A]\in\I(\cA)$.

Put $\{f_\alpha: \alpha<\kappa\}$ as an  enumeration of elements of $\cF'$ such that $f_0$ is the identity function. Note that $\cF'$ consists of injections and that  $|\cF'|\leq |[\cF]^{<\omega}|<\mathfrak{t}$.

We will now construct inductively a sequence $(T_\alpha)_{\alpha<\kappa}$ such that
\begin{enumerate}
\item $T_{\alpha}\in [X]^\omega$ for every $\alpha<\kappa$;
\item if $\alpha>\beta$ then $T_\alpha\subseteq^* T_\beta$;
%\item $T_0\cap A\in {\fin}$ for all $A\in\cA$;
\item for each $\alpha<\kappa$, we have $f^{-1}_\alpha[T_\alpha]\in \I(\cA)$  or 
$$f^{-1}_\alpha[T_\alpha] \in \I \text{ and }  f^{-1}_\alpha[T_\alpha]\cap A\in {\fin} \text{ for all } A\in\cA ;$$
\item for every $\beta,\gamma\leq \alpha<\kappa$, if $f^{-1}_\beta[T_\alpha]\not\in\I(\cA)$ and  $f^{-1}_\gamma[T_\alpha]\not\in\I(\cA)$ then  $$ f^{-1}_\beta[T_\alpha]\cap f^{-1}_\gamma[T_\alpha]\in{\fin} \quad \text{or} \quad  f_\beta^{-1}\restriction T_\alpha=^* f_\gamma^{-1}\restriction T_\alpha.$$
\end{enumerate}

Fix $\alpha<\kappa$, and assume we have constructed $T_\beta$ for $\beta<\alpha$.
Since $|\alpha|<\mathfrak{t}$, there exists a set $S\in [X]^\omega$ such that $S\subseteq^* T_\beta$ for all $\beta<\alpha$.

%If there exists an infinite $S_0\subseteq S $ such that $f_\alpha^{-1}[S_0]\in\I(\cA)$ then put $T_\alpha=S_0$.
If there exists an infinite $C\subseteq S $ such that $f_\alpha^{-1}[C]\in\I(\cA)$ then put $S_0=C$.

%If there is no such $S_0$ then $f^{-1}_\alpha[S]\cap A\in {\fin}$ for all $A\in\cA$ as otherwise, should there be $A_0\in\cA$ such that $f^{-1}_\alpha[S]\cap A_0\not\in {\fin}$ then $S_0=f_\alpha\left[f^{-1}_\alpha[S]\cap A_0 \right]$ would be an infinite subset of $S$ with $f_\alpha^{-1}[S_0]\subseteq A_0\in\I(\cA)$.
If there is no such $C$ then $f^{-1}_\alpha[S]\cap A\in {\fin}$ for all $A\in\cA$ as otherwise, should there be $A_0\in\cA$ such that $f^{-1}_\alpha[S]\cap A_0\not\in {\fin}$ then $C=f_\alpha\left[f^{-1}_\alpha[S]\cap A_0 \right]$ would be an infinite subset of $S$ with $f_\alpha^{-1}[C]\subseteq A_0\in\I(\cA)$.

In this case, %we put as $T_\alpha$ an infinite subset of $S$ such that $f^{-1}[T_\alpha]\in \I$. To see that we can find such, 
notice that we have $f_\alpha^{-1}[S]\not\in\I(\cA)$, thus $f_\alpha^{-1}[S]\not\in{\fin}$. Since $\I$ is tall, there exists an infinite $D\subseteq f_\alpha^{-1}[S]  $ such that $D\in\I$. We put $S_0=f_\alpha[D]$. Notice that $S_0$ is an infinite subset of $S$ such that $f_\alpha^{-1}[S_0]\in \I$.

Next, let  $(\alpha+1)\times (\alpha+1)=\{(\beta_\xi,\gamma_\xi): \xi<\lambda \}$ with $\lambda =|(\alpha+1)\times (\alpha+1)|<\mathfrak{t}$. We will now construct inductively a sequence $(S_\xi)_{\xi<\lambda}$ such that 
$S_\xi\subseteq^* S_\eta$ for $\eta<\xi$
and 
if $f^{-1}_{\beta_\xi}[S_\xi]\not\in\I(\cA)$ and  $f^{-1}_{\gamma_\xi}[S_\xi]\not\in\I(\cA)$ then  
$$ f^{-1}_{\beta_\xi}[S_{\xi+1}]\cap f^{-1}_{\gamma_\xi}[S_{\xi+1}]\in{\fin} \quad \text{or} \quad f_{\beta_\xi}^{-1}\upharpoonright S_{\xi+1}=^* f_{\gamma_\xi}^{-1}\upharpoonright S_{\xi+1}. $$
We have already defined $S_0$. 

For a given $\xi<\lambda$, suppose we have defined $S_\xi$. We have two cases.

In the first case, there exists $C\in [S_\xi]^\omega$ such  that $f^{-1}_{\beta_\xi}[C]\cap f^{-1}_{\gamma_\xi}[C]\in{\fin}$. Then we put $S_{\xi+1}=C$.

In the second case, for all $C\in [S_\xi]^\omega$ we have $f^{-1}_{\beta_\xi}[C]\cap f^{-1}_{\gamma_\xi}[C]\not\in{\fin}$. Since $f^{-1}_{\beta_\xi}$ and $f^{-1}_{\gamma_\xi}$ are injections, we can realize that we have $f_{\beta_\xi}^{-1}\upharpoonright S_{\xi}=^* f_{\gamma_\xi}^{-1}\upharpoonright S_{\xi}$. We put $S_{\xi+1}=S_\xi$.

For a limit $\eta<\lambda$, suppose we have defined all $S_\xi$ for $\xi < \eta$. Then we find such $S_\eta$ that $S_\eta\subseteq^* S_\xi$ for all $\xi<\eta$ -- we can find such, because $\eta<\lambda<\mathfrak{t}$.

Finally, we put as $T_\alpha$ an infinite subset of $S_0$ such that $T_\alpha\subseteq^* S_\xi$ for all $\xi<\lambda$.

Now that we have constructed the sequence $(T_\alpha)$, we  may find the set $T\in[X]^\omega$ such that $T\subseteq^* T_\alpha$ for every $\alpha<\kappa$. Since $\I$ is tall, we may also demand that $T\in\I$.

As a next step we define the family 
$$\cB=\cA\cup \{T\}\cup\left\{f^{-1}_\alpha[T]: \alpha<\kappa,\, f^{-1}_\alpha[T]\not\in\I(\cA),\, \forall \beta<\alpha \,( f^{-1}_\alpha[T]\not=^*f^{-1}_\beta[T])  \right\}.$$ 

It is easy to see that $\cA\subseteq\cB$, $\cB\subseteq\I$ and $|\cB|\leq |\cA| +1 + \kappa<\mathfrak{t}$. Moreover, $T\in \cB\cap [X]^\omega$.  

Next, notice that $T\subseteq X$, thus $T\cap A\in {\fin}$ for every $A\in\cA$. Moreover, notice that if  $f^{-1}_\alpha[T]\in \cB$, $f^{-1}_\beta[T]\in \cB$ are two distinct sets  for some $\alpha>\beta$ then $f^{-1}_\alpha[T]\cap f^{-1}_\beta[T]\in {\fin}$ as $f^{-1}_\alpha[T_\alpha]\cap f^{-1}_\beta[T_\alpha]\in {\fin}$ and $T\subseteq^*T_\alpha$. Additionally, if $f^{-1}_\alpha[T]\in\cB$ then $f^{-1}_\alpha[T]\cap A\in {\fin}$ for every $A\in\cA$ as  $f^{-1}_\alpha[T_\alpha]\cap A\in {\fin}$ and $T\subseteq^*T_\alpha$. 
Moreover, $f_0^{-1}[T]=T\not\in\I(\cA)$ as  $T$ is an infinite subset of $X$ and  $A\cap X\in{\fin}$ for every $A\in\cA$. 
Therefore, $\cB$ is an almost disjoint family. 

To finish the proof, we need to show that $\cB$ respects $\cF$. Pick any $f\in\cF$ and $B\in\cB$. Obviously, $f=f_\alpha$ for some $\alpha<\kappa$.

If $B\in\cA$ then $f^{-1}[B]\in\I(\cA)\subseteq\I(\cB)$, because $\cA$ respects $\cF$. 

If $B=T$  then either $f^{-1}[T]\in\I(\cA)\subseteq\I(\cB)$ or  $f^{-1}[T]\in\cB\subseteq\I(\cB)$ or there exists $\beta<\alpha$ such that $f^{-1}[T]=^*f^{-1}_\beta[T]\in\cB$, thus $f^{-1}[T]\in\I(\cB)$. 

If $B=f_\beta^{-1}[T]$ for some $\beta<\kappa$, then there exists some $\gamma<\kappa$ such that $f_\gamma^{-1}=f_\alpha^{-1}\circ f_\beta^{-1}$. It follows that $f^{-1}[B]=f_\gamma^{-1}[T]$, which belongs to $\I(\cB)$ by the reasoning presented in the previous case. %because, just like in previous case, either $f_\gamma^{-1}[T]\in\I(\cA)\subseteq\I(\cB)$ or  $f_\gamma^{-1}[T]\in\cB\subseteq\I(\cB)$ or there exists $\delta<\gamma$ such that $f_\gamma^{-1}[T]=^*f^{-1}_\delta[T]\in\cB$, thus $f_\gamma^{-1}[T]\in\I(\cB)$. 
\end{proof}

\begin{lemma}[{\cite[Lemma~3.11]{MR2017358}}]
\label{lem:hr2}
Let $\cA$ be an almost disjoint family that respects $\cF\subseteq\omega^\omega$. 
Assume also that $|\cA|<\mathfrak{a}$,   and $X\not\in\I(\cA)$.
Then there exists an injection $f:\omega\rightarrow X$ such that $\cA$ respects $\cF\cup\{f\}$.    
\end{lemma}
    
\begin{theorem}
\label{thm:MAD-K-homo}
Assume $\mathfrak{t}=\continuum$. 
If  $\I$ is a tall ideal then there exists a MAD family $\cA\subseteq \I$ such that $\I(\cA)$ is $K$-homogeneous. 
\end{theorem}
\begin{proof}
First, let $\{X_\alpha:\alpha<\continuum\}$ be the enumeration of elements of $[\omega]^\omega$. We will inductively construct sequences $(\cA_\alpha)_{\alpha<\continuum}$ and $(\cF_\alpha)_{\alpha<\continuum}$ such that for every $\alpha<\continuum$
\begin{enumerate}
\item\label{point-i} $\cA_\alpha$ is an AD-family such that $\cA_\alpha\subseteq\I$;
\item $\cF_\alpha\subseteq\omega^\omega$ consists of injections;
\item if $\beta<\alpha$ then $\cA_\beta\subseteq \cA_\alpha$ and $\cF_\beta\subseteq \cF_\alpha$;
\item\label{point-iv} $\cA_0$ is a partition of $\omega$ into infinitely many infinite sets belonging to $\I$;
\item $\cF_0=\emptyset$;
\item\label{point-vi} $|\cA_\alpha|<\continuum$ and $|\cF_\alpha|<\continuum$;
\item $\cA_\alpha$ respects $\cF_\alpha$;
\item\label{point-viii}\label{point-mad} there exists $A\in \cA_{\alpha+1}$ such that $A\cap X_\alpha\not\in{\fin}$;
\item\label{point-ix}\label{point-khom}  if $X_\alpha\not\in\I(\cA_{\alpha+1})$ then there is $f\in\cF_{\alpha+1}$ with $\ran(f)\subseteq X_\alpha$.
\end{enumerate}

Fix $0<\alpha<\continuum$ and assume we have constructed $\cA_\beta$ and $\cF_\beta$ for all $\beta<\alpha$.

If $\alpha$ is a limit ordinal then we put $\cA_\alpha=\bigcup_{\beta<\alpha} \cA_\beta$ and $\cF_\alpha=\bigcup_{\beta<\alpha} \cF_\beta$. Since $\mathfrak{t}$ is a regular cardinal (see e.g.~\cite[Proposition~6.4]{MR2768685}), we obtain $|\cA_\alpha|<\continuum$ and $|\cF_\alpha|<\continuum$.

Suppose now that $\alpha=\beta+1$. 
If $X_\beta\in\I(\cA_\beta)$,  we put $\cA_\alpha=\cA_\beta$ and $\cF_\alpha=\cF_\beta$.
If $X_\beta\not\in\I(\cA_\beta)$ then by Lemma~\ref{lem:hr1} we can find an AD-family $\cA_{\alpha}\subseteq\I$ such that $\cA_\beta\subseteq \cA_{\alpha}$, $|\cA_\alpha|<\continuum$, $\cA_\alpha$ respects $\cF_\beta$ and there exists an infinite $A\in\cA_\alpha$ such that $A\subseteq X_\beta$.
In both cases, the families $\cA_\alpha$ and $\cF_\beta$ satisfy items (\ref{point-i})-(\ref{point-viii}). Now, we will define $\cF_\alpha$ to satisfy item (\ref{point-ix}).

If $X_\beta\in\I(\cA_\alpha)$, we put $\cF_\alpha=\cF_\beta$. In the other case, by Lemma~\ref{lem:hr2} there exists an injection $f\in\omega^\omega$ with $\ran(f)\subseteq X_\beta$ such that $\cA_\alpha$ respects $\cF_\beta\cup\{f\}$. Then we put $\cF_\alpha=\cF_\beta\cup\{f\}$. 

Now that we have constructed sequences $(\cA_\alpha)_{\alpha<\continuum}$ and $(\cF_\alpha)_{\alpha<\continuum}$, let 
$$\cA=\bigcup_{\alpha<\continuum} \cA_\alpha \quad \text{and} \quad \cF=\bigcup_{\alpha<\continuum} \cF_\alpha.$$ 
Observe that $\cA$ respects $\cF$. Indeed, take any $A\in\cA$ and $f\in\cF$. The sequences $(\cA_\alpha)$ and $(\cF_\alpha)$ are increasing, thus we can find $\alpha<\continuum$ such that $A\in\cA_{\alpha}$ and $f\in\cF_\alpha$. Since $\cA_\alpha$ respects $\cF_\alpha$, we obtain $f^{-1}[A]\in\I(\cA_\alpha)\subseteq\I(\cA)$. 

By items~(\ref{point-i}), (\ref{point-iv}) and (\ref{point-mad}) of the construction, $\cA$ is a MAD family such that  $\cA\subseteq\I$.

To finish the proof, we need to show that $\I(\cA)$ is $K$-homogeneous. Take $X\not\in\I(\cA)$. Then there exists $\alpha<\continuum$ such that $X=X_\alpha$ and clearly $X_\alpha\not\in \I(\cA_{\alpha+1})$. By item~(\ref{point-khom}) of the construction there exists $f\in\cF_{\alpha+1}\subseteq\cF$ with $\ran(f)\subseteq X_\alpha$. Since $\cA$ respects $\cF$, we know that for every $A\in\cA$ we have $f^{-1}[A]\in\I(\cA)$.
Therefore, this $f$ is a witness for $\I(\cA)\restriction X\leq_K \I(\cA)$.  
\end{proof}

\begin{corollary}
\label{cor:non-Borel-intersection-of-matrix}
Assume $\mathfrak{t}=\continuum$.
There exists a non-Borel ideal that is the intersection of some matrix ideals.
\end{corollary}
\begin{proof}
Take $\I(\cA)$ from Theorem~\ref{thm:MAD-K-homo} with $\I=\I_d$. By \cite[Proposition~4.6]{MR491197}, $\I(\cA)$ is a non-Borel ideal.  Since $\cA\subseteq\I_d$, we have $\I(\cA)\subseteq\I_d$, thus $\I(\cA)\leq_K\I_d$.
Moreover, by $K$-homogeneity of $\I(\cA)$, for every $A\not\in\I(\cA)$ we obtain 
$$\I(\cA)\upharpoonright A\leq_K \I(\cA)\leq_K \I_d,$$
thus, by Theorem~\ref{thm:GMV-char}, $\I(\cA)$ is an intersection of matrix ideals.
\end{proof}

%----------------------------------------------------------------------
% Bibliography 
%----------------------------------------------------------------------

\bibliographystyle{amsplain}
\bibliography{nonpathology}

\end{document}